\pdfoutput=1
\documentclass{amsart}

\usepackage[margin=16mm]{geometry}

\usepackage{epsfig}
\usepackage{amssymb,amsmath,amsthm,ytableau}
\usepackage{amscd,pb-diagram} 
\usepackage[all]{xypic}
\usepackage[numbers]{natbib}
\usepackage[colorlinks]{hyperref}
\usepackage{MnSymbol}

\newtheorem{theorem}{Theorem}[section]

\newtheorem{lemma}{Lemma}[section]

\newtheorem{proposition}{Proposition}[section]

\theoremstyle{definition}
\newtheorem{definition}{Definition}[section]
\newtheorem{remark}{Remark}[section]

\newcommand{\Z}{\mathbb{Z}}
\newcommand{\R}{\mathbb{R}}

\newcommand{\X}{\mathfrak X}

\newcommand{\Span}[1]{\left\langle#1\right\rangle}

\DeclareMathOperator{\SO}{\mathsf{SO}}
\DeclareMathOperator{\OOO}{\mathsf{O}}
\newcommand{\CO}{\widetilde{\mathsf{CO}}}

\DeclareMathOperator{\GL}{\mathsf{GL}}

\DeclareMathOperator{\Aff}{Aff}

\DeclareMathOperator{\Hess}{Hess}

\newcommand{\A}{\mathcal{A}}

\newcommand{\E}{\mathcal{E}}

\newcommand{\CC}{\mathcal{C}}

\newcommand{\TT}{\mathcal{T}}

\newcommand{\NN}{\mathcal{N}}

\newcommand{\Th}{^\textrm{th}}

\newcommand{\Rd}{^\textrm{rd}}
\newcommand{\Nd}{^\textrm{nd}}

\begin{document}

\title[Affine--invariant  PDEs and Fubini--Pick]{
Third--order affine--invariant (systems of) PDEs in two independent variables as vanishing of the Fubini--Pick invariant}

 \author{Dmitri~Alekseevsky}
   \address{Institute for Information Transmission Problems, B. Karetny
per. 19, 127051, Moscow (Russia) and University of Hradec Kralove, Rokitanskeho 62,
Hradec Kralove 50003 (Czech Republic).}
\email{dalekseevsky@iitp.ru}

 \author{Gianni Manno}
   \address{Dipartimento di Matematica ``G. L. Lagrange'', Politecnico di Torino, Corso Duca degli Abruzzi, 24, 10129 Torino, Italy.}
    \email{giovanni.manno@polito.it}
 \author{Giovanni Moreno}
 \address{Department of Mathematical Methods in Physics,
 Faculty of Physics, University of Warsaw,
ul. Pasteura 5, 02-093 Warszawa, Poland}
 \email{giovanni.moreno@fuw.edu.pl}
 


\maketitle

\begin{abstract}
In this paper we study $3\Rd$ order (system of) PDEs in two independent variables $x,y$ and one unknown function $u$  that are invariant with respect to the group of affine transformation $\Aff(3)$ of $\R^3=\{(x,y,u)\}$. After proving their relationship with the Fubini--Pick invariant, we derive the aforementioned PDEs by using a general method introduced in \cite{alekseevsky2020general}, which sheds light on some of their geometrical properties.
\end{abstract}

\setcounter{tocdepth}{1}



\section*{Introduction}

It is well known that the only $\Aff(3)$--invariant second--order scalar PDE is the one given by $\det(u_{ij})=0$: a proof can be found, e.g.,  in~\cite{MR2324300} (see also the references therein, as well as~\cite{Ushakov2000TheEG}   for a discussion of the general solution);  remarkably, the   contact geometry   of this equation can  be described in terms of its characteristics, see~\cite{MR2985508,MR2383541,MR2503974}.  
In this paper we show how to construct (non--trivial systems of) $\Aff(3)$--invariant \emph{third--order} PDEs in one unknown function and two independent variables by using a general method developed in \cite{alekseevsky2020general,AMMv2_2022}. The geometric character of the proposed construction will be duly emphasised: first,  by highlighting  the relation of the so--obtained PDEs with the Fubini--Pick invariant, then by studying  the compatibility conditions (if a   \emph{system} of $\Aff(3)$--invariant  of third--order PDEs is obtained) or  the geometry of  characteristics (if a   \emph{scalar} $\Aff(3)$--invariant  third--order PDE is obtained) and,  finally, by clarifying the overall role played by the Blaschke metric.\par 

A \emph{statistical manifold} is a (pseudo)--Riemannian manifold $(M,h)$ equipped with a connection $\nabla$, such that $\nabla h$ is symmetric, see~\cite{Kurose1994ONTD}:  it is worth mentioning that the Blaschke metric, the Blaschke connection, and the   Fubini--Pick $(0,3)$--tensor, all satisfy the axioms of a statistical manifold. Therefore, our geometric departing point can be  framed in this larger context.\par 
The paper is organized  as follows.\par
In Section~\ref{secAffDiffGeo} we focus on the basics of  the affine geometry of surfaces of $\R^3$ and the affine structure of jet spaces. Contextually,  we give the definition  of a $G$--invariant (systems of) PDEs.\par
In Section~\ref{secFidHyp}, after recalling the main results of \cite{alekseevsky2020general}, we explain  how to construct $G$--invariant (systems of) PDEs by using the aforementioned  method.\par
In Section~\ref{sec:construction.aff.inv.pdes} we obtain  the same  $\Aff(3)$--invariant (system of) PDE from two different viewpoints. The first one  is based on the method used in~\cite{MR2406036}: $G$--invariant PDEs can be obtained as singular orbits of the prolonged action of $G$ on suitable jet spaces (Section~\ref{sec:oliveri}). The second one shows that the same (system of) PDEs can be obtained by employing the affine second fundamental form $h^\xi$ associated to a transversal vector field $\xi$ and the corresponding affine $(0,3)$--tensor $C^\xi$, rather than by using the Blaschke metric (Section~\ref{secTensDer}).\par
In Section~\ref{secCasAff} we use the method outlined in Section \ref{secFidHyp} to construct $\Aff(3)$--invariant (systems of) PDEs. The idea is to (equivariantly) extend certain submanifolds defined in a fibre of the jet space of order $3$ to the whole jet space by means of the affine action of $\Aff(3)$, showing that only two cases are possible, accordingly to the signature of the Hessian matrix. Indeed, if the Hessian is non-degenerate (this is always the case if we assume that the affine second fundamental form is non--degenerate), we get an $\Aff(3)$--invariant system of two PDEs of third order; on the other hand, if the determinant of the Hessian is negative, we obtain also a pair of $\Aff(3)$--invariant scalar PDEs of third order: if these two scalar PDEs are considered as a system, the result is the same  $\Aff(3)$--invariant system of PDEs mentioned before.\par
In Section~\ref{sec:det.hess.maggiore.zero} we focus on the $\Aff(3)$--invariant system of PDEs by computing its compatibility conditions. In particular, we show that a solution $u=u(x,y)$ to this system described a  conic  in both the planes $(x,u)$ and  $(y,u)$.\par
In Section~\ref{sec:det.hess.minore.zero} we focus on the $\Aff(3)$--invariant scalar PDEs: we show that their characteristic distribution degenerates into a $3$--dimensional vector sub-distribution of the $5$--dimensional Cartan distribution (i.e., the contact distribution of order $2$) on $J^2$. We finally prove  that such a $3$--dimensional distribution completely characterizes the $\Aff(3)$--invariant scalar PDEs, in the sense that their solutions are Legendrian submanifolds of $J^1$ whose prolongation on $J^2$ non--trivially interesect the aforementioned distribution.

\subsection*{Notations and conventions}

We denote by $\odot$ the symmetric product between tensors.
%
In the case of differential forms, we will often omit the symbol $\odot$, for instance:
\begin{equation}\label{eqn:dxidxj.sym}
dx^idx^j:=dx^i\odot dx^j=\frac12(dx^i\otimes dx^j + dx^j\otimes dx^i)\,.
\end{equation}
We denote by $S^kV$ the $k\Th$ \emph{symmetric tensor power} of $V$.
For a subset $U\subseteq V$ of the vector space $V$, symbol $\Span{U}$ denotes the linear span of $U$ in $V$;   we simply write $\R v$ for the linear span of $U=\{v\}$. If a group $G$ acts on a set $S$, then  we let $G\cdot A:=\{g\cdot a\mid g\in G\, ,a\in A\}$  for any $A\subseteq S$.
The module of vector fields on a manifold $M$ is denoted by $\X(M)$.
A system of coordinates on $\R^3$ will be denoted by $(x,y,u)$,  unless otherwise specified.
The Einstein convention on repeated indices will be used, unless otherwise specified.

 \subsection*{Acknowledgments}
D.~Alekseevsky gratefully acknowledges support by the Grant Basis-Foundation Leader n.  22-7-1-34-1. G.~Manno gratefully acknowledges support by the project ``Connessioni
proiettive, equazioni di Monge-Amp\`ere e sistemi integrabili'' (INdAM),
``MIUR grant Dipartimenti di Eccellenza 2018-2022 \linebreak (\texttt{E11G18000350001})'', ``Finanziamento alla Ricerca'' \texttt{53\_RBA17MANGIO} and \texttt{53\_RBA21MANGIO},  and
PRIN project 2017 ``Real and Complex Manifolds: Topology, Geometry and
holomorphic dynamics'' (code \texttt{2017JZ2SW5}).  G.~Manno is a member of\linebreak GNSAGA
of INdAM.  G.~Moreno is supported by   the Polish National Science Center  project “Complex contact
manifolds and geometry of secants”, \texttt{2017/26/E/ST1/00231}.

\section{Preliminaries}\label{secAffDiffGeo}

\subsection{Basics on affine geometry of surfaces \cite{An-Min:2015aa,nomizu1994affine}}\label{subAffCaseR3}


Let $S$ be a surface of $\mathbb{R}^3$ and $\xi$ be a transversal field to it. We denote by $h^\xi$ the \emph{affine second fundamental form} associated to $\xi$, i.e., the bilinear form defined by
\begin{equation}\label{eqn:affine.II.form}
D_XY=\nabla^\xi_XY + h^\xi(X,Y)\xi\, ,\quad X,Y\in\X(S)\, ,
\end{equation}
where $D$ is the Levi-Civita connection of $\mathbb{R}^3$ and $\nabla^\xi_XY$ is the tangential component.
The rank of $h^\xi$ is independent of the choice of the transversal field $\xi$, so that
a surface $S\subset\R^3$ is called \emph{non--degenerate} if the rank of $h^\xi$ is equal to $2$. From now on, unless otherwise specified, we shall work only with non--degenerate surfaces.
Let $\omega\in\Omega^3(\mathbb{R}^3)$ be a $D$--parallel volume element: the \emph{induced volume form} on $S$, associated with  $\xi$, is the differential 2--form $\omega^\xi$ defined by
\begin{equation*}
\omega^\xi(X,Y):={\omega}(X,Y,\xi)\,,\quad X,Y\in\X(S)\, ,
\end{equation*}
and it is easy to see that
\begin{equation*}
\nabla^\xi(\omega^\xi)=\tau^\xi\otimes\omega^\xi
\end{equation*}
for some $1$--differential form $\tau^\xi$.
The Codazzi equation for $h^\xi$,
$$
(\nabla^\xi_X h^\xi)(Y,Z)+\tau^\xi(X)h^\xi(Y,Z)=(\nabla^\xi_Y h^\xi)(X,Z)+\tau^\xi(Y)h^\xi(X,Z)\,,\quad X,Y,Z\in\X(S)\,,
$$
implies that the following $(0,3)$--tensor on $S$, that we call the \emph{affine $(0,3)$--tensor associated with  $\xi$}, is symmetric:
\begin{equation}\label{eqDefPick}
    C^\xi(X,Y,Z):=(\nabla^\xi_X h^\xi)(Y,Z)+\tau^\xi(X)h^\xi(Y,Z)\,,\quad X,Y,Z\in\X(S)\,.
\end{equation}
Recalling that we are assuming the surface $S$ to be non--degenerate, there exists a unique, up to sign, transversal field $\A$, called the \emph{affine normal}, such that $\tau^\A=0$ and  $\omega^\A$ coincides with the volume element associated to $h^\A$.
The corresponding connection $\nabla:=\nabla^\A$, the  $(0,2)$--tensor $h:=h^\A$ and the $(0,3)$--tensor $C:=C^\A$,
\begin{equation}\label{eqDefPick.vero}
    C(X,Y,Z)=(\nabla_X h)(Y,Z)\,,\quad X,Y,Z\in\X(S)\,,
\end{equation}
are called, respectively, the \emph{Blaschke connection}, the \emph{Blaschke metric} and the \emph{Fubini--Pick cubic form}.
The contraction of the Fubini-Pick form $C$ with the Blaschke metric, i.e.,
\begin{equation}\label{eqPickContr}
C^{ijk} C_{ijk} =    h^{i_1i_2}h^{j_1j_2}h^{k_1k_2}C_{i_1 j_1 k_1}C_{i_2 j_2 k_2}\, ,
\end{equation}
is  a function called the \emph{Fubini-Pick invariant}. Hence,  the  surface  $S$  with the (pseudo)--Riemannian metric $h$ and  the  connection $\nabla$ is  a  statistical manifold as the Fubini--Pick cubic form $C=\nabla h$ is symmetric. In the context of statistical manifolds such tensor is called also the \emph{Amari--Chentsov tensor}.




\subsection{Jet spaces and PDEs \cite{KrasilshchikLychaginVinogradov1986,MR989588,Vinogradov1988a}}\label{sec:jets}

From now on, $M$ will be a smooth manifold of dimension $3$
and  $S \subset M$ an embedded  surface of $M$, unless specified otherwise. Locally, in an appropriate system of local coordinates $(x^1, x^2,u)$ of a neighborhood of $p\in S$, the surface $S$ is represented as $S=S_f=\{u=f(x^1,x^2)\}$ where $f$ is a smooth function of the variables $(x^1,x^2)$.

\smallskip
We denote by $J^\ell$ the \emph{space of $\ell$--jets of surfaces} of $M$, i.e.,
%
\begin{equation*}
J^\ell:=\bigcup_{p\in M}\left\{[S]^\ell_{p}\mid \text{$S$ is a surface of $M$ passing through $p$}\right\}\, ,
\end{equation*}
where $[S]^\ell_{p}$ denotes the equivalence class of surfaces having with $S$ a contact of order $\ell$ at $p$.\par

The natural map
$
j^\ell:p\in S \mapsto [S]^\ell_{p}\in J^\ell
$
is called the \emph{$\ell\Th$--jet extension}  of $S$
and symbol
\begin{equation}\label{eq:jet.ext}
S^{(\ell)}:=j^\ell(S)
\end{equation}
denotes its image.

\smallskip
The  coordinates
 $(x^1,x^2,u)$   defines local coordinates
$$(x^1,x^2, u, u_i, u_{ij}, \cdots,  u_{i_1 i_2 \cdots i_{\ell}})$$
of $J^{\ell}$, such that the prolongation  $S_f^{(\ell)}$ of the  surface  $S_f$ is locally described by
$$
\left(x^1,x^2,f(x^1,x^2),\frac{\partial f}{\partial x^i},\frac{\partial^2 f}{\partial x^i x^j},\dots,\frac{\partial^\ell f}{\partial x^{i_1} x^{i_2}\cdots x^{i_\ell}}\right)\, ,
$$
which shows  that $S_f^{(\ell)}$ is a $2$--dimensional submanifold of $J^\ell$.
%
%

The natural projections
\begin{equation}\label{eqEqPiKaKaMenoUno}
\pi_{\ell,m}:J^{\ell}\stackrel{}{\longrightarrow} J^{m}\,,\quad [S]^\ell_{p}\longmapsto [S]^m_{p}\, ,\quad \ell>m\, ,
\end{equation}
define a tower of bundles. 
%
In what follows,  a point $[S]^\ell_{p}\in J^\ell$ will be often denoted simply by $a^\ell$: 
accordingly,
\begin{equation}\label{eqn:fibre.of.J}
J^\ell_{a^m}:=\pi_{\ell,m}^{-1}(a^m)\,,\quad a^m\in J^m\,,
\end{equation}
denotes the fiber of $\pi_{\ell,m}$ over $a^m$. The following definition will come in handy.
\begin{definition}\label{defTautBundle}
For $\ell\geq 1$, the \emph{tautological} vector bundle $\TT^\ell \subset \pi_{\ell,\ell-1}^* (TJ^{\ell-1})$   is the bundle over $J^{\ell}$ given by
\begin{equation*}
\TT^\ell:=\left\{ (a^\ell,v)\in J^\ell\times TJ^{\ell-1}\,\,|\,\,v\in  T_{a^{\ell-1}}S^{(\ell-1)}\right\}\,,
\end{equation*}
i.e., the fiber $\TT^\ell_{a^\ell}$ over the point $a^\ell=[S]^{\ell}_{a^0}$ is
$
\TT^\ell_{a^\ell}=T_{a^{\ell-1}}S^{(\ell-1)}\,.
$
\end{definition}

It is well known that $\pi_{\ell,\ell-1}:J^{\ell}\longrightarrow J^{\ell-1}$, for $\ell\geq 2$, are affine bundles modeled on the   vector bundle
\begin{equation}\label{eqn:vertical.tang.bundle}
\ker\left(d\pi_{\ell,\ell-1}\right)\,
\end{equation}
 on $J^\ell$, 
where $d\pi_{\ell,\ell-1}:TJ^\ell \longrightarrow TJ^{\ell-1}$: the  \emph{contact distribution of order $\ell$} on $J^\ell$  can be then defined as:
\begin{equation}\label{eqn:higher.contact}
\CC^{\ell}:=(d\pi_{\ell,\ell-1})^{-1}\left(\TT^{\ell}\right)\, ,
\end{equation}
%
see~\cite{MR861121,MR2352610,MR722524}; the next result is also classical, see~\cite{MR861121,MR989588}.
\begin{lemma}\label{lemma.ker.vector}
There is an identification 
\begin{equation*}
\ker\left(d\pi_{\ell,\ell-1}\right)\simeq \pi_{\ell,1}^* (S^\ell\TT^* \otimes \NN)\,,
\end{equation*}
where $\TT:=\TT^1$ and $\NN$ 
is the vector bundle on $J^1$ whose fibres are defined by $\NN_{[S]^1_p}:=\frac{T_{p}M}{T_{p}S}$.
\end{lemma}

\begin{remark}\label{rem.Taylor}This  affine bundle structure    will be playing a key role in Section~\ref{secFidHyp} below, since  it allows to regard the difference
between two points $[S]^{\ell}_p$ and $[\widetilde{S}]^{\ell}_p$ of the fibre $J^\ell_{a^{\ell-1}}$ as an element of $S^\ell T_{p}^*S\otimes \frac{T_{p}M}{T_{p}S}$:   
for instance, with $\ell=2$, Lemma~\ref{lemma.ker.vector} locally reads
\begin{equation}\label{eqn:S2p.meno}
    [S]^2_p - [\widetilde{S}]^2_p=(x^i,u,u_i,u_{ij})-(x^i,u,u_i,\widetilde{u}_{ij})\simeq (u_{ij}-\widetilde{u}_{ij})\partial_{u_{ij}}\simeq  \frac{2}{1+\delta_{ij}} (u_{ij}-\widetilde{u}_{ij}) dx^i dx^j   \otimes\partial_u \,.
\end{equation}
\end{remark}
A \emph{system of $m$ PDEs of order $\ell$ in $2$ independent variables and one unknown (or dependent) variable} on a manifold $M=J^0$ is a $m$--codimensional sub--bundle $\E$ of $J^\ell$. Taking into account \eqref{eqn:fibre.of.J}, we define the fiber of $\E\subset J^\ell$ over $a^m\in J^m$, $m<\ell$:
\begin{equation}\label{eqn:fibre.of.E}
\E_{a^m}:=\E\cap J^\ell_{a^m}\,.
\end{equation}
A \emph{solution} to $\E$  is  a surface $S$ of $M$ such that its $\ell\Th$ jet extension $S^{(\ell)}$ lies in $\E$, cf.   \eqref{eq:jet.ext}.

\smallskip
A local diffeomorphism $\phi$ of $M$
naturally acts on all the jet spaces via its prolongation
\begin{equation}\label{eqn:prol.jet.spaces}
\left[S\right]_{p}^\ell\in  J^\ell  \longrightarrow [\phi(S)]_{\phi(p)}^\ell\in J^\ell\,
\end{equation}
to $J^\ell$.  
This allows to speak of  \emph{$G$-invariance} for a system of PDEs $\E\subset J^\ell$, where $G$ is a connected Lie group acting on $M$: the system $\E$ is   \emph{$G$-invariant}  if the subset $\E$ is preserved by $G$. Let $X$ be an  infinitesimal generator of $G$: its prolongation $X^{(\ell)}$ to $J^\ell$ is obtained by prolonging its local flow; therefore, if $\E$ is   {$G$-invariant}, then  $X^{(\ell)}$  is tangent to $\E$. It is worth recalling the prolongation formula for $X=X_0\partial_u + X^i\partial_{x^i}$:
\begin{equation}\label{eqn:X.r}
X^{(\ell)}=X_0\partial_u + X^i\partial_{x^i}  +  \sum_{|\sigma|=1}^\ell X_{\sigma}\partial_{u_{\sigma}}\,,
\end{equation}
where
$\sigma$ is a multi--index of length $|\sigma|$ defined by
$\sigma=(\sigma_1,\dots,\sigma_s)$, $s\leq \ell$, $\sigma_i\in\{1,\dots,n\}$, $\sigma_1\leq\sigma_2\leq\cdots\leq\sigma_s$,  and the component $X_{\sigma}$ are recursively defined by
$$
X_{\sigma,j}:=D_{x^j}(X_\sigma) - u_{\sigma,i}D_{x^j}(X^i)\,,
$$
where $D_{x^j}$ is the  operator  of total derivation with respect to $x^j$.



\section{A general method to construct invariant PDEs
on homogeneous manifolds}\label{secFidHyp}

We review now one of the main result of \cite{alekseevsky2020general}. 
Let us assume that $J^0=M$ is a homogeneous manifold, i.e., $J^0=G/H$ for a  connected Lie group $G$, and  choose a point $o\in J^0$,
as well as $o^{\ell}\in J^\ell$, projecting onto $o$: we can then  consider, $\forall\,\ell\geq 2$, the fibre $J^{\ell}_{o^{\ell-1}}$ as a vector  space with origin $o^{\ell}$. The natural lift of an element $g\in G$ on each  $\ell$--jet space $J^{\ell}$ is given by~\eqref{eqn:prol.jet.spaces}, which  in this case reads
$$
   g :   o^{\ell}= [S]^{\ell}_{o}\in J^{\ell}   \longrightarrow g\cdot o^\ell:=[g\cdot S]^{\ell}_{g(o)} \in J^{\ell}\,.
$$
We will denote by $H^{(\ell)}$ the stabilizer of $o^{\ell}$ in $G$, that is
$$
H^{(\ell)}:=G_{o^{\ell}}\, ,
$$
and by
\begin{equation*}
\tau :  H^{(k-1)} \longrightarrow \Aff(J^k_{o^{k-1}})
\end{equation*}
 the affine  action  of $H^{(k-1)}$   on  the  fibre  $J^k_{o^{k-1}}$. Then
 \begin{equation}\label{eqn:Wk.orbit}
W^k:=\tau(H^{(k-1)})\cdot o^{k}\subset J^k_{o^{k-1}}\,
\end{equation}
defines an affine subspace.\par 
%

\medskip\noindent
If we assume that there exists a point $o^k\in J^k$, with $k\geq 2$, such that\par \medskip
\begin{itemize}
  \item[\textbf{(A1)}] the   orbit $\check{J}^{k-1} := G\cdot o^{k-1}  \subset  J^{k-1}$ is open,
  \end{itemize}
  \medskip
then
$$ \pi_{k,k-1}: \check{J}^k:= \pi_{k, k-1}^{-1}(\check{J}^{k-1}) \longrightarrow\check{J}^{k-1} = G/H^{(k-1)}$$
is a homogeneous affine fibre  bundle defined by the affine representation $\tau$ of the stability group $H^{(k-1)}$ in the fiber $J^k_{o^{k-1}}$.

\begin{proposition}[\cite{alekseevsky2020general}]
Let $M=G/H$ be a homogeneous manifold and let $o^k$ be a point    satisfying  {\normalfont (A1)}. 
Any $\tau(H^{k-1})$--invariant hypersurface $\Sigma \subset J^k_{o^{k-1}}$ extends to a $G$--invariant hypersurface $\mathcal{E} := G \cdot\Sigma\subset \check{J}^k$, which is a $G$--invariant PDE of order $k$.
\end{proposition}
The assumption $(A2)$ below reduces the classification of the aforementioned invariant PDEs to the description of hypersurfaces, that are invariant under a linear group. We denote by $L_{H^{(k-1)}} = H^{(k-1)}_{o^k}$ the   subgroup of the affine group $\tau (H^{(k-1)})$  stabilizing  $o^k$, which  can be considered as the linear part of the affine group $\tau (H^{(k-1)})$.
\par \medskip
\begin{itemize}
  \item[\textbf{(A2)}]
The affine group $\tau(H^{(k-1)})$ is the  semidirect product
$
\tau(H^{(k-1)}) = L_{H^{(k-1)}} \rtimes T_{W^k}
$,
\end{itemize}
where $T_{W^k}$ is the group of translations or, equivalently, the orbit $W^k$ defined by~\eqref{eqn:Wk.orbit} is a vector space.\par 
A homogeneous manifold $M=G/H$, such that both assumptions (A1) and (A2) are satisfied has been referred to as a \emph{$k$--admissible manifold} in~\cite{AMMv2_2022}: to simplify the exposition of the present paper, we will use the stronger assumption of the existence of a \emph{fiducial surface} in $M$. 
%
\begin{definition}\label{defMAIN}
A  surface $S \subset  M$ is  called  a \emph{fiducial surface} (of order $k$ for the group $G$ at the origin $o$) if
\begin{itemize}
\item $o\in S$,
    \item $S$ is homogeneous  with respect to a  subgroup  of $G$,
    \item $o^{k-1}=[S]^{k-1}_o$   satisfies (A1),
    \item $o^{k}=[S]^{k}_o$ satisfies (A2), that is the affine subspace $W^k=\tau(H^{(k-1)})\cdot o^k$ given by~\eqref{eqn:Wk.orbit} is a linear subspace.
\end{itemize}
\end{definition}
We can now state one of the main result of \cite{alekseevsky2020general}, duly recast in the context of systems of PDEs. Recall that   $L_{H^{(k-1)}}$ denotes  the linear part of $\tau(H^{(k-1)})$.
\begin{theorem}\label{thMAIN1}
There is  a natural  $1$--$1$   correspondence  between   $L_{H^{(k-1)}}$--invariant $m$-codimensional submanifolds $\overline{\Sigma}$ of $J^k_{o^{k-1}}/W^k$
and  $G$--invariant systems of $m$ PDEs  $\E_{\Sigma}:=G\cdot\Sigma$  of  $J^k$, where $\Sigma\subset J^k_{o^{k-1}}$ is the pre-image of $\overline{\Sigma}$ (see  Figure~\ref{cap.1}).
\end{theorem}
\begin{figure}[h]
\centering
\includegraphics[scale=0.65]{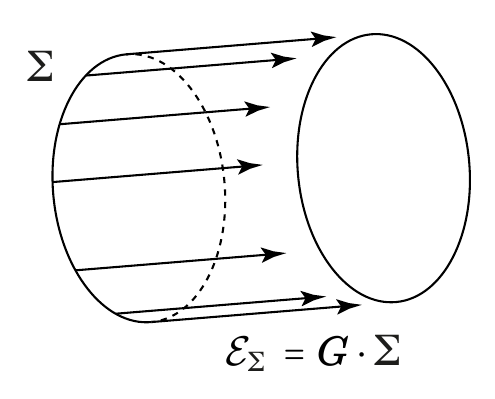}
\caption{Let $m=1$, then a $G$-invariant PDE $\E_{\Sigma}$ is obtained  as the   $G$--extension $\E_{\Sigma}=G\cdot\Sigma$ of the cylinder $\Sigma$ over a $L_{H^{(k-1)}}$--invariant hypersurface $\overline{\Sigma} \subset p(J^k_{o^{k-1}})$, i.e., of the pre--image   $\Sigma=p^{-1}(\overline{\Sigma})$ of  $\overline{\Sigma} $ in the vector space $p(J^k_{o^{k-1}}) = J^k_{o^{k-1}}/W^k$, where $p$ is the canonical projection.}\label{cap.1}
\end{figure}
%

%




\section{Affine invariant third--order PDEs}\label{sec:construction.aff.inv.pdes}
Let $(x,y,u)$ be local coordinates on $\R^3$. Then, according to Section~\ref{sec:jets}, they define local coordinates on any $k$-jet space:
\begin{equation}\label{eqDefJk}
J^k:=\{(x,y,u,u_x,u_y,u_{xx},u_{xy},u_{yy},u_{xx},u_{xxx},u_{xxy},u_{xyy},u_{yyy},u_{xxxx},\dots)\}\,.
\end{equation}

\subsection{Lie symmetry method for obtaining  the $\Aff(3)$--invariant PDE}\label{sec:oliveri}
%
%
%
%
%
%
%
%
The Lie algebra $\mathfrak{aff}(3)$ of the vector fields on $J^3$ associated with the action of the affine group $\Aff(3)$  is spanned by the prolongations $X_i^{(3)}$ to $J^3$ of $X_i$, where $X_i$ is one of the twelve generators
\begin{equation}\label{eqn:aff.vectors}
\begin{array}{llllllllllll}
\partial_x\,, & \partial_y\,, & \partial_u\,, & x\partial_x\,, & x\partial_y\,, & x\partial_u\,, &
y\partial_x\,, & y\partial_y\,, & y\partial_u\,, & u\partial_x\,, & u\partial_y\,, & u\partial_u\, ,
\end{array}
\end{equation}
with $i=1, \dots, 12$, cf.  \eqref{eqn:X.r}. The 
assignment  
\begin{eqnarray*}
     J^3\ni a&\longmapsto &
\mathcal{D}_a:=\textrm{Span}\langle X_1^{(3)}|_a\,,\,X_2^{(3)}|_a\,,\ldots\,, X_{12}^{(3)}|_a\rangle \subset T_aJ^3
\end{eqnarray*}
defines a distribution on $J^3$, which is of dimension $12=\dim J^3$ everywhere, except for a closed subset $\E\subset J^3$. Since $\E$ is $\Aff(3)$--invariant, it leads to a  $\Aff(3)$--invariant  (system of) PDEs, 
%
%
as clarified in  the next Proposition (see~\cite{MR2406036} for more details). 


\begin{proposition}\label{propPrimaFormInvarianza}
The subset $\E\subset J^3$ described by
\begin{equation}\label{eqPDE_completa}
\E:\,\,(u_{xx}u_{yy}-u_{xy}^2)\cdot F=0\,,
\end{equation}
where $F$ is
\begin{align}\label{eqFormulaSospirata}
F&=6 u_{xx} u_{xxx} u_{xy} u_{yy}
   u_{yyy}-6 u_{xx} u_{xxx} u_{xyy}
   u_{yy}^2-18 u_{xx} u_{xxy} u_{xy}
   u_{xyy} u_{yy}+12 u_{xx} u_{xxy}
   u_{xy}^2 u_{yyy} \\
  &-6 u_{xx}^2 u_{xxy}
   u_{yy} u_{yyy}
   +9 u_{xx} u_{xxy}^2
   u_{yy}^2-6 u_{xx}^2 u_{xy} u_{xyy}
   u_{yyy}+9 u_{xx}^2 u_{xyy}^2
   u_{yy}+u_{xx}^3 u_{yyy}^2-6 u_{xxx}
   u_{xxy} u_{xy} u_{yy}^2\nonumber\\
 &+12 u_{xxx}
   u_{xy}^2 u_{xyy} u_{yy}-8 u_{xxx}
   u_{xy}^3 u_{yyy}+u_{xxx}^2 u_{yy}^3\,, \nonumber
\end{align}
is  $\Aff(3)$--invariant.
\end{proposition}
As we will see later, the aforementioned subset $\E$ encompasses  all possible $\Aff(3)$--invariant  (systems of) PDEs. 
\begin{remark}
A solution of $u_{xx}u_{yy}-u_{xy}^2=0$ is also a solution of $F=0$,  where $F$ is given by \eqref{eqFormulaSospirata}. This can be seen by considering the differential consequences of $u_{xx}u_{yy}-u_{xy}^2=0$, i.e., $ u_{xxx}u_{yy}+u_{xx}u_{xyy}-2u_{xy}u_{xxy}=0$ and $u_{xxy}u_{yy}+u_{xx}u_{yyy}-2u_{xy}u_{xyy}=0$.
%
%
%
%
%
\end{remark}

\begin{proposition}\label{prop:Fubini_uguale_zero}
Let $u_{xx}u_{yy}-u_{xy}^2\neq 0$ and let $F$ be given by \eqref{eqFormulaSospirata}. Then $F=0$ if and only if the Fubini--Pick invariant \eqref{eqDefPick.vero} vanishes.
\end{proposition}
\begin{proof}
It is a straightforward computation.
\end{proof}

%
%
%

\subsection{Codimension and smoothness of (systems of) $\Aff(3)$--invariant PDEs}
%

An immediate consequence of \eqref{eqPDE_completa}  is that $\E$
is   the union of the Monge--Amp\`ere equation
$
u_{xx}u_{yy}-u_{xy}^2=0
$
and of the third--order PDE described by $F=0$, where $F$ is given by \eqref{eqFormulaSospirata}.
It is now   convenient to define
\begin{equation}\label{eqDefJ3check}
{J}^k_+:=\{ \text{points of \eqref{eqDefJk} such that } \det(\Hess(u))> 0\}\, ,\quad {J}^k_-:=\{ \text{points of \eqref{eqDefJk} such that } \det(\Hess(u))< 0\}\, ,
\end{equation}
with $k\geq 2$: indeed, in order to study the $\Aff(3)$--invariant PDE $\E$ given by~\eqref{eqPDE_completa}, we can focus on
%
 \begin{equation}\label{eqn:E.meno}
\E_+:=\E\cap {J}^3_+\,,\quad \E_-:=\E\cap {J}^3_-\,.
\end{equation}

\begin{proposition}\label{PropPrimaProposizioneSpezzamento}
$\E_+$ is a smooth submanifold of $J^3$ of codimension 2, whereas $\E_-$
is the union
\begin{equation}\label{eqn:E1.unito.E2}
    \E_-=\E_-^1\cup\E_-^2\,
\end{equation}
of two smooth hypersurfaces   of $J^3$.
\end{proposition}
\begin{proof}
It can be proved by direct computations: by assuming $u_{yy}>0$, the function $F$ defined by \eqref{eqFormulaSospirata} can be brought to the form
\begin{multline}
\left(\sqrt{u_{yy}^3}u_{xxx}+\frac{1}{\sqrt{u_{yy}^3}} (3u_{xx}u_{xy}u_{yy}u_{yyy}-3u_{xx}u_{xyy}u_{yy}^2-3u_{xxy}u_{xy}u_{yy}^2-4u_{xy}^3u_{yyy}+6u_{xy}^2u_{xyy}u_{yy}) \right)^2\\
+ \frac{1}{u_{yy}^3} (u_{xx}u_{yy}-u_{xy}^2) ( u_{xx}u_{yy}u_{yyy}-3u_{xxy}u_{yy}^2-4u_{xy}^2u_{yyy}+6u_{xy}u_{xyy}u_{yy}  )^2\, ,\label{eqSplittingBrutale}
\end{multline}
that   is the sum (resp., difference) of two squares in the case when $\det(\Hess(u))$ is positive (resp., negative), i.e.,
\begin{multline}\label{eqn:union.PDEs}
\E_-^1\cup\E_-^2:\,\,\bigg(\sqrt{-\det(\Hess(u))} \big(u_{xx}u_{yy}u_{yyy}-3u_{xxy}u_{yy}^2-4u_{xy}^2u_{yyy}+6u_{xy}u_{xyy}u_{yy} \big)
\\
+ \big(-3u_{xx}u_{xy}u_{yy}u_{yyy}+3u_{xx}u_{xyy}u_{yy}^2-u_{xxx}u_{yy}^3
+3u_{xxy}u_{xy}u_{yy}^2+4u_{xy}^3u_{yyy}-6u_{xy}^2u_{xyy}u_{yy}\big) \bigg)
\\
\cdot\bigg(\sqrt{-\det(\Hess(u))} \big(u_{xx}u_{yy}u_{yyy}-3u_{xxy}u_{yy}^2-4u_{xy}^2u_{yyy}+6u_{xy}u_{xyy}u_{yy} \big)
\\
- \big(-3u_{xx}u_{xy}u_{yy}u_{yyy}+3u_{xx}u_{xyy}u_{yy}^2-u_{xxx}u_{yy}^3
+3u_{xxy}u_{xy}u_{yy}^2+4u_{xy}^3u_{yyy}-6u_{xy}^2u_{xyy}u_{yy}\big) \bigg)=0\,.
\end{multline}
A similar reasoning works in the case  $u_{yy}<0$ and leads to the same PDEs.
\end{proof}

\begin{remark}\label{rem:Giovanni.1}

%
%
The intersections
$
\E_{a^2_\pm}=\E\cap J^3_{a^2_\pm}
$
 are cut out by the functions
$F|_{J^3_{a^2_\pm}}$
 respectively, where $F$ is given by \eqref{eqFormulaSospirata}. Immediate computations shows that
%
\begin{eqnarray}
F|_{J^3_{a^2_+}}&=&(3 u_{xxy} - u_{yyy})^2 + (3u_{xyy} -  u_{xxx})^2\, ,\label{eqSplitting1}\\
F|_{J^3_{a^2_{-}}}&=& (u_{xxx}-3u_{xxy}+3u_{xyy}-u_{yyy})(-u_{xxx}-3u_{xxy}-3u_{xyy}-u_{yyy})\, .\label{eqSplitting2}
\end{eqnarray}
%
Directly from \eqref{eqSplitting1} it follows that $\E_{a^2_+}$ is the \textit{intersection} of the two hypersurfaces
\begin{equation}\label{eqGiUnoGiDue}
3 u_{xxy} - u_{yyy}=0\,,\quad 3u_{xyy} - u_{xxx}=0\,
\end{equation}
of $J^3_{a^2_+}$, 
whereas  it follows from \eqref{eqSplitting2} that  $\E_{a^2_-}$ is the \textit{union} of the two hypersurfaces
\begin{equation}\label{eqEffeUnoEffeDue}
u_{xxx}-3u_{xxy}+3u_{xyy}-u_{yyy}=0\, ,\quad -u_{xxx}-3u_{xxy}-3u_{xyy}-u_{yyy}=0\, .
\end{equation}
Proposition \ref{PropPrimaProposizioneSpezzamento} can be also proved  by extending, respectively, submanifolds \eqref{eqGiUnoGiDue} and \eqref{eqEffeUnoEffeDue}
to the whole of $J^2_{\pm}$ by means of $\Aff(3)$.
\end{remark}

\subsection{A tensorial derivation of the (system of) $\Aff(3)$--invariant PDEs}\label{secTensDer}

%
\begin{proposition}\label{prop:Cxi.proportional}
Let $S=\{u=f(x,y)\}$ be a surface of $\R^3$, $h^\xi$ as in \eqref{eqn:affine.II.form} and $C^\xi$ the cubic form \eqref{eqDefPick}.
\begin{itemize}
    \item If $\det(\Hess(u))>0$, then the condition of  existence of a $1$-form $\alpha\in\Omega^1(S)$ such that
    \begin{equation}\label{eqn:Cxi.proportional.+}
C^\xi=\alpha\odot h^\xi \,,
\end{equation}
is independent of the choice of the transversal field $\xi$.
    \item If $\det(\Hess(u))<0$, then the condition of  existence  of a $1$-form $\alpha\in\Omega^1(S)$ and a symmetric $2$-form $\beta$ on $S$ such that
    \begin{equation}\label{eqn:Cxi.proportional.-}
C^\xi=\alpha\odot h^\xi + \beta\odot\theta\,,
\end{equation}
where $\theta\in\Omega^1(S)$ is a non-zero 1--form such that $(h^{\xi})^{-1}(\theta,\theta)=0$,
is independent of the choice of the transversal field $\xi$.
\end{itemize}
\end{proposition}
\begin{proof}
    By   direct computations.
\end{proof}
%
%
\begin{theorem}\label{th:1}
Let $S=\{u=f(x,y)\}$ be a surface of $\R^3$, $h^\xi$ as in \eqref{eqn:affine.II.form} and $C^\xi$ the cubic form \eqref{eqDefPick}. Then the function $f(x,y)$ is a solution to the equation ${\E}_+$ (resp., ${\E}_-$)  (cf. \eqref{eqn:E.meno}) if and only if condition ~\eqref{eqn:Cxi.proportional.+} (resp., \eqref{eqn:Cxi.proportional.-})  of Proposition \ref{prop:Cxi.proportional} is satisfied.
\end{theorem}
\begin{proof}

To begin with, let us consider the case when condition \eqref{eqn:Cxi.proportional.+} holds true.
Taking into account the local expression of $C^\xi$ and that $h^\xi$ is proportional to the Hessian matrix of $u=f(x,y)$, it is easy to see that condition \eqref{eqn:Cxi.proportional.+} reads
\begin{equation}\label{eqn:H3.prop.H2.H1.1}
(u_{xxx}dx^3+3u_{xxy}dx^2dy+3u_{xyy}dxdy^2+u_{yyy}dy^3)=(\alpha_1dx+\alpha_2dy)(u_{xx}dx^2+2u_{xy}dxdy+u_{yy}dy^2)\,,
\end{equation}
for some functions $\alpha_1,\alpha_2$.
A direct computation shows that, after eliminating $\alpha_1$ and $\alpha_2$,  condition \eqref{eqn:H3.prop.H2.H1.1} becomes a system that locally describes $\E_+$:
\begin{equation}\label{eqn:sys.to.be.studied}
\E_+:
    \left\{
    \begin{array}{l}
    u_{xx}^2u_{yyy}-3u_{xx}u_{yy}u_{xxy} + 2u_{xy}u_{yy}u_{xxx}=0\,,
    \\
    u_{yy}^2u_{xxx} -3u_{xx}u_{yy}u_{xyy} + 2u_{xy}u_{xx}u_{yyy} =0\,.
    \end{array}
    \right.
\end{equation}
System \eqref{eqn:sys.to.be.studied} coincides with the system coming from \eqref{eqSplittingBrutale} in the case when $\det(\Hess(u))>0$.

\smallskip\noindent
The case when condition \eqref{eqn:Cxi.proportional.-} holds true can be treated analogously: indeed, such condition reads
\begin{multline}\label{eqn:H3.prop.H2.H1}
(u_{xxx}dx^3+3u_{xxy}dx^2dy+3u_{xyy}dxdy^2+u_{yyy}dy^3)=(\alpha_1dx+\alpha_2dy)(u_{xx}dx^2+2u_{xy}dxdy+u_{yy}dy^2) +
\\
(\beta_{11}dx^2+2\beta_{12}dxdy+\beta_{22}dy^2)\left( \left(u_{xy}\pm \sqrt{-u_{xx}u_{yy}+u_{xy}^2}\right)dx+u_{yy}dy\right)\,,
\end{multline}
for some functions $\alpha_1,\alpha_2,\beta_{11},\beta_{12},\beta_{22}$. 
A direct computation shows that after eliminating $\alpha_1,\alpha_2,\beta_{11},\beta_{12},\beta_{22}$, one obtains two equations (see~\eqref{eqn:union.PDEs}) locally describing $\E_-=\E_-^1\cup\E_-^2$,
where $\E_-^1$ corresponds to the PDE obtained by equating to zero the expression given by the first two lines of~\eqref{eqn:union.PDEs}, whereas $\E_-^2$ corresponds to the PDE obtained by equating to zero the expression  given by the last two lines of~\eqref{eqn:union.PDEs}.
\end{proof}
Theorem \ref{th:1} gives a geometric interpretation of solutions of  equations  \eqref{eqn:Cxi.proportional.+} and \eqref{eqn:Cxi.proportional.-}.
Equation \eqref{eqn:Cxi.proportional.+} means that the cubic form $C^{\xi}$  is  divided  by the metric $h^{\xi}$, whereas equation \eqref{eqn:Cxi.proportional.-} means that the  remainder of the division of $C^{\xi}$ by $h^{\xi}$ is a decomposable cubic form, more precisely a product of an isotropic 1-form and a quadratic form.

\begin{remark}
Theorem \ref{th:1} might also be proved 
by working in a specific fibre of $\pi_{3,2}$ and then acting by $\Aff(3)$, see   Remark~\ref{rem:Giovanni.1}. We will check  this only in the case when $\det(\Hess(u))>0$, since the case of a negative $\det(\Hess(u))$ is virtually the same: if we fix the point
$$
a^2_+=(0,0,0,0,0,1,0,1)\in {J}^2_+\,,
$$
then condition \eqref{eqn:H3.prop.H2.H1.1} becomes
 \begin{equation*}
(u_{xxx}dx^3+3u_{xxy}dx^2dy+3u_{xyy}dxdy^2+u_{yyy}dy^3)=(\alpha_1 dx+\alpha_2 dy)(dx^2+dy^2)\, ,
\end{equation*}
i.e.,
$$
u_{xxx}=\alpha_1\,,\,\,3u_{xxy}=\alpha_2\,,\,\,3u_{xyy}=\alpha_1\,,\,\,u_{yyy}=\alpha_2\,.
$$
By eliminating $\alpha_1$ and $\alpha_2$, we obtain the system \eqref{eqGiUnoGiDue}.
\end{remark}

\section{$\Aff(3)$--invariant PDEs as extensions of invariant subsets of the fiber of $J^3$}\label{secCasAff}

\subsection{Preliminary results needed to apply the main theorem}

We focus now on the three--dimensional affine space $\R^3=\Aff(3)/\GL(3)$, regarded as a homogeneous manifold $J^0=G/H$, with the origin $a^0:=(0,0,0)$. In order to apply the algorithm outlined in Theorem \ref{thMAIN1}, we need fiducial hypersurfaces for the affine group $\Aff(3)$, and we need to compute the stabilizer subgroups $H^{(k)}$.
\begin{lemma}\label{lem.fid.parab.hyperb}
The paraboloid $S_+$ (resp. the hyperboloid $S_-$)
\begin{equation}\label{eq:fiducial.hyper.Gianni}
S_{\pm}:=\left\{u=Q_\pm(x,y)\right\}\, ,
\end{equation}
where
\begin{equation}\label{eqQuPiuMeno}
    Q_\pm(x,y)=\frac12(x^2\pm y^2)\, ,
\end{equation}
 is a   fiducial hypersurface of order 2 for the affine group $\Aff(3)$ at origin $a^0$.
\end{lemma}
\begin{proof}
A straightforward computation shows that both $S_\pm$ fulfill all conditions of Definition~\ref{defMAIN}, see also~\cite{AMMv2_2022}.
\end{proof}
%
We will denote by
\begin{equation}\label{eqDefPiPiuMeno}
a^k_\pm:=[S_\pm]_{(0,0)}^k\,,
\end{equation}
where $S_\pm$ is given by \eqref{eq:fiducial.hyper.Gianni}, the origins associated to the fiducial hypersurfaces $S_\pm$.
\begin{lemma}[\cite{AMMv2_2022}]\label{lemStrutturaSottogruppiCasoAffine}
For $k=0,1,2,3$, the   subgroups $H^{(k)}$ of $G=\Aff(3)$ that stabilize  the origins $a_\pm^{k}$,   are:
\begin{eqnarray}
  H=H^{(0)}&=&   \GL(3)\, ,\nonumber\\
H^{(1)} &=&  (\R^2\rtimes\GL(2))\times\R^\times \, ,\label{eqH1casoAff}\\
H^{(2)} &=&  (\R^2\rtimes\OOO(1,1))\times\R^\times\, ,\quad a_-^{2}\in J^2_- \, ,\label{eqH2casoAff}\\
&=&  (\R^2\rtimes\OOO(2))\times\R^\times\, ,\quad a_+^{2}\in J^2_+\, ,\nonumber\\
H^{(3)} &=&  \OOO(1,1)\times\R^\times\, ,\quad a_-^{3}\in J^3_-\,\nonumber\\
&=&  \OOO(2)\times\R^\times\, ,\quad a_+^{3}\in J^3_+\, .\nonumber
\end{eqnarray}
\end{lemma}
%
%
%

%

\begin{remark}
Since $\OOO(2)\times\R^\times$ is a $\Z_2$--covering of the  conformal group $\mathsf{CO}(2)$, we will employ the notation $\CO(2):=\OOO(2)\times\R^\times$; an analogous symbol $\CO(1,1)$ will be used for the Lorentzian signature case.
\end{remark}


Now we describe the linear subspace $W_\pm^3$   that is associated with  the fiducial hypersurface $S_\pm$ introduced in Lemma~\ref{lem.fid.parab.hyperb} above, cf. Definition~\ref{defMAIN}.\par 
If we start from the fiducial surface  $S_{\pm}$, then the (open) orbit of   $a^2_{\pm}$ turns out to be ${J}^{2}_{\pm}$ and the linear space $W^3_{\pm}$ is the (two--dimensional) linear subspace  of   the (four--dimensional) space $S^3\R^{2\ast}$ made of cubic forms on $\R^2$  that are proportional to $Q_{\pm}(x,y)$ via a linear factor:
$$
W^3_{\pm}=\R^{2\,\ast}\odot \R Q_{\pm}\,.
$$
%
%
%
%
Taking into account Remark~\ref{rem.Taylor},
Theorem \ref{thMAIN1} claims  that    the only $\Aff(3)$--invariant third--order conditions  we can impose on a surface $S\subset\R^3$ at a point $p\in S$ are:
 \begin{itemize}
\item[$(+)$] if $p$ is such that $[S]_p^2$ is in the same $\Aff(3)$--orbit as the paraboloid $S_+$: the condition is that, by taking the difference $[S]_p^3-[S_+]_p^3$  modulo $W^3_+$, the result lands in a $\CO(2)$--invariant proper subset $\overline{\Sigma}$ in $S^3\R^{2\,\ast}/W^3_+$;
\item[$(-)$] if $p$ is such that $[S]_p^2$ is in the same $\Aff(3)$--orbit as the hyperboloid $S_-$: the condition is  that, by taking the difference $[S]_p^3-[S_-]_p^3$ modulo $W^3_-$,
 the result lands in a $\CO(1,1)$--invariant proper subset $\overline{\Sigma}$ in $S^3\R^{2\,\ast}/W^3_-$;
\item[$(0)$] if $p$ is not any of the previous ones: the condition coincides with the prolongation to $J^3$ of the second--order PDE  $\det u_{ij}=0$.
\end{itemize}

Now we are ready  to apply Theorem~\ref{thMAIN1} to the case of $G=\Aff(3)$ acting on $J^0=\R^3$.

%
%
%


\begin{theorem}\label{corCasoAff}
Let $Q_\pm$  be the quadratic forms and   $S_\pm$ be the fiducial hypersurfaces introduced in Lemma~\ref{lem.fid.parab.hyperb} above.  Let
\begin{equation}\label{eqS03.traceless}
S_0^3\R^{2\,\ast}:=\frac{S^3\R^{2\,\ast}}{\R^{2\,\ast}\odot{\R Q_\pm}}\,.
\end{equation}
Then, for any $\CO(2)$--invariant (resp., $\CO(1,1)$--invariant) subset
\begin{equation*}
\overline{\Sigma}\subset S_0^3\R^{2\,\ast}\,,
\end{equation*}
the condition that the difference $[S]_p^3-[S_\pm]_p^3$ modulo $\R^{2\,\ast}\odot{\R Q_\pm}$  be $\overline{\Sigma}$--valued defines  the   $\Aff(3)$--invariant (system of) third--order PDE(s) $\E_\Sigma\subset J^3_\pm $.   
\end{theorem}





We recall that  the graphs of the functions $Q_{\pm}$ given by~\eqref{eqQuPiuMeno}  are fiducial hypersurfaces in the sense of Definition \ref{defMAIN}. 
In particular,   they describe two open $\Aff(3)$--orbits in $J^2$, that we have denoted by $J^2_+$ and $J^2_-$, 
respectively.\par
%



In order to use  Theorem \ref{corCasoAff} to produce examples of $\Aff(3)$--invariant (systems of) PDEs, 
we need   to study $\CO(2)$-- or $\CO(1,1)$--invariant subset $\overline{\Sigma}\subset S_0^3\R^{2\,\ast}$: it turns out that such an invariance 
implies that $\overline{\Sigma} $ can only take two forms, as it will be proved below: 
the aforementioned two possibilities will correspond exactly to the $\Aff(3)$--invariant PDEs $\E_+$ and $\E_-$, introduced earlier, see~\eqref{eqn:E.meno}.



We warn the reader that, while keep using coordinates $\{x,y,u\}$ in $\R^3$, the same symbols $\{x,y\}$ will also denote a basis of $\R^{2\,\ast}$: accordingly,
\begin{equation}\label{eqStCoordS2R3}
    S^3\R^{2\,\ast}=\Span{x^3, 3x^2y, 3xy^2, y^3}\, ,
\end{equation}
and this will be the standard basis of $S^3\R^{2\,\ast}$. Moreover, since the point $a_\pm^{3}$ (see~\eqref{eqDefPiPiuMeno}) allows to identify $J^3_{a_\pm^{2}}$ with $S^3\R^{2\,\ast}$, the standard coordinates $\{u_{xxx},u_{xxy},u_{xyy},u_{yyy}\}$ of  $J^3_{a^{2}_{\pm}}$ turn out to be the dual coordinates   to \eqref{eqStCoordS2R3}.


\subsection{The case when the fiducial hypersurface is the paraboloid $S_+$}\label{subsecRieman}
In this section we set $Q=Q_+$. It will be convenient to identify   the  space
\begin{equation}\label{eqTrFreeCubQUO}
S_0^3\R^{2\,\ast}=\frac{S^3\R^{2\,\ast}}{\R^{2\,\ast}\odot \R Q}=\frac{\Span{x^3, 3x^2y, 3xy^2, y^3}}{\Span{xQ, yQ}}=\frac{S^3\R^{2\,\ast}}{\R^{2\,\ast}\odot \R Q}=\frac{\Span{x^3, 3x^2y, 3xy^2, y^3}}{\Span{x^3+xy^2,yx^2+y^3}}
\end{equation}
with its subspace
\begin{equation}
 \Span{x^3, y^3}\, .\label{eqTrFreeCubSUB}
\end{equation}
Keeping this identification in mind, the same symbol $S_0^3\R^{2\,\ast}$ will be used for both \eqref{eqTrFreeCubQUO} and \eqref{eqTrFreeCubSUB}; it is also necessary to introduce a basis
\begin{equation}\label{eqxieta}
    \{\xi,\eta\}
\end{equation}
of $(S_0^3\R^{2\,\ast})^\ast$ that is dual to $\{x^3,y^3\}$. \par
%
%
To study the $\CO(2)$--invariant subsets $\overline{\Sigma}$ of the two--dimensional space $S_0^3\R^{2\,\ast}$, we have to clarify first the $\CO(2)$--action on it.
\subsubsection{The action of $\CO(2)$ on $\Span{x^3, y^3}$}\label{secCasoRiem}
Since
\begin{equation}\label{eqDecomposizioneDelCxxxo}
\CO(2)=\OOO(2)\times\R^\times\, ,
\end{equation}
we can consider separately the action of a rotation, a reflection and a dilation. To begin with, if
\begin{equation*}
R_t:=\left(\begin{array}{cc}\cos(t) & \sin(t) \\-\sin(t) & \cos(t)\end{array}\right)\in \SO(2)
\end{equation*}
is a rotation,
then
\begin{align*}
R_t^*(x^3)&=(R_t^*(x))^3\\
&=(x\cos(t)-y\sin(t))^3\\
&=x^3\cos^3(t)-3x^2y\cos^2(t)\sin(t)+3xy^2\cos(t)\sin^2(t)-y^3\sin^3(t)\\
&=x^3\cos^3(t)+3 y^3\cos^2(t)\sin(t)-3x^3\cos(t)\sin^2(t)+y^3\sin^3(t)\\
&=\cos(3t)x^3+ \sin(3t)y^3\, ,\\
R_t^*(y^3)&=-\sin(3t) x^3+\cos(3t)y^3\, .
\end{align*}
In other words, if we let a $2\times 2$ matrix act on $S_0^3\R^{2\,\ast}$ by identifying the latter with $\R^2$ via the basis  $\{x^3, y^3\}$, then $R_t\in\SO(2)$  is the rotation $R_{-3t}$; similarly, the reflection  $x\longrightarrow -x$ corresponds to the reflection $x^3\longrightarrow -x^3$ and the scaling $(x,y)\longrightarrow \lambda (x, y)$ will correspond to the scaling $(x^3,y^3)\longrightarrow \lambda^3(x^3,y^3)$. In view of \eqref{eqDecomposizioneDelCxxxo}, we can conclude that the $\CO(2)$--action on  $S_0^3\R^{2\,\ast}$ can be identified with the standard $\CO(2)$--action  on $\R^2$, that possesses only two invariant subsets: $\{0\}$ and $\R^2$ itself.\par
Since we are not interested in trivial PDEs, the only choice for the subset $\overline{\Sigma}$ in Theorem \ref{corCasoAff} is $\overline{\Sigma}=\{0\}$.

\subsubsection{The system of PDEs associated with $\overline{\Sigma}=\{0\}$} According to Theorem \ref{corCasoAff}, the (system of) PDEs associated with $\overline{\Sigma}=\{0\}$ describes surfaces $S$ of $\R^3$ such that $[S]_p^3-[S_+]^3_p$ lands in $\Sigma=W^+=\R^{2\,\ast}\odot \R Q$, that is, a codimension--two linear subspace of $S^3\R^{2\,\ast}$. By employing the aforementioned dual coordinates
$\{u_{xxx},u_{xxy},u_{xyy},u_{yyy}\}$,
 we see that the subspace $\R^{2\,\ast}\odot \R Q$ is cut out precisely  by the two linear equations
\begin{equation}\label{eqSystAff3}
u_{xxx}-3u_{xyy}=0\, ,\quad u_{yyy}-3u_{xxy}=0\, .
\end{equation}
%
Since the system \eqref{eqSystAff3} can be recast as a unique equation $(u_{xxx}-3u_{xyy})^2+(u_{yyy}-3u_{xxy})^2=0$, the fiber $(\E_\Sigma)_{a_+^{2}}\subset J^3_{a_+^{2}}$ of the $\Aff(3)$--invariant PDE $\E_\Sigma$ constructed by Theorem \ref{corCasoAff} can be given by
\begin{equation*}
    (\E_\Sigma)_{a_+^{2}}=\{f=0\}\subset J^3_{a_+^{2}}\, ,
\end{equation*}
where
\begin{equation*}
f(u_{xxx},u_{xxy},u_{xyy},u_{yyy})=(u_{xxx}-3u_{xyy})^2+(u_{yyy}-3u_{xxy})^2\,
\end{equation*}
is a function on $J^3_{a_+^{2}}$.
The  $H^{(2)}$--invariant subset $(\E_\Sigma)_{a_+^{2}}$ of ${J}^3_{a_+^{2}}$ can be extended  to an $H^{(1)}$--invariant subset
\begin{equation}\label{eqPrimoSpalmamentoCasoRiem}
(\E_\Sigma)_{a_+^{1}}=H^{(1)}\cdot (\E_\Sigma)_{a_+^{2}}\,
\end{equation}
  of $J^3_{a^1_+}$.
In fact, the group $H^{(1)}$ in formula~\eqref{eqPrimoSpalmamentoCasoRiem} can be replaced by its factor $\GL(2)$, that acts naturally on $J^2_{a_+^1}=\GL(2)\cdot   Q$. The latter is the orbit in $S^2\R^{2\,\ast}$ of $Q$, i.e., (one half) the Euclidean squared norm, whose associated matrix is $\tfrac{1}{2}I_2$:
\begin{equation*}
J^2_{a_+^1}=\{ A\cdot\left(\tfrac{1}{2}I_2\right)\cdot A^t\mid A\in\GL(2)\}=\{ A\cdot A^t\mid A\in\GL(2)\}\, .
\end{equation*}
Therefore,  any point $(u_{xx},u_{xy},u_{yy})\in J^2_{a_+^1}$ can be brought to the form $A\cdot A^t$ for some  $A\in\GL(2)\subset H^{(2)}$. It follows that
\begin{equation*}
(\E_\Sigma)_{(u_{xx},u_{xy},u_{yy})}=A\cdot (\E_\Sigma)_{a_+^{2}}\subset   J^3_ {(u_{xx},u_{xy},u_{yy})}
\end{equation*}
must be cut out by the equation
\begin{equation*}
F(u_{xx},u_{xy},u_{yy},u_{xxx},u_{xxy},u_{xyy},u_{yyy}):=A^{-1\,\ast}(f)=0\, ,
\end{equation*}
where $A^{-1}$ is regarded as a diffeomorphism from the fiber $J^3_ {(u_{xx},u_{xy},u_{yy})}$ to the fiber $J^3_{a^1_+}$, by applying the prolongation formula~\eqref{eqn:prol.jet.spaces} to the diffeomorphism $A^{-1}$ of $J^0$.
Then, 
the same function $F$, that has no dependency upon $x,y,u,u_x,u_y$, cuts out the whole equation $\E_\Sigma$ in $J^3_+ $.\par
Obtaining $F$ out of $f$ is not complicated:  it turns out that
\begin{align*}
A^{-1\,\ast}(u_{xxx}) &=\det(A)^{-3}(a_{22}^3 u_{xxx}-3 a_{22}^2 a_{12} u_{xxy}+3 a_{22}
   a_{12}^2 u_{xyy}+a_{12}^3 \left(-u_{yyy}\right))\,,\\
   A^{-1\,\ast}(u_{xxy}) &=\det(A)^{-3}(-a_{22}^2 a_{12} u_{xxx}+\left(2 a_{22} a_{12}^2+a_{11}
   a_{22}^2\right) u_{xxy}+\left(-a_{12}^3-2 a_{11} a_{22}
   a_{12}\right) u_{xyy}+a_{11} a_{12}^2 u_{yyy})\,,\\
   A^{-1\,\ast}(u_{xyy}) &=\det(A)^{-3}(a_{12}^2 a_{22} u_{xxx}+\left(-a_{12}^3-2 a_{11} a_{22}
   a_{12}\right) u_{xxy}+\left(a_{22} a_{11}^2+2 a_{12}^2
   a_{11}\right) u_{xyy}-a_{12} a_{11}^2 u_{yyy})\,,\\
   A^{-1\,\ast}(u_{yyy}) &=\det(A)^{-3}(-a_{12}^3 u_{xxx}+3 a_{12}^2 a_{11} u_{xxy}-3 a_{12}
   a_{11}^2 u_{xyy}+a_{11}^3 u_{yyy})\,,
\end{align*}
whence
\begin{align*}
F&=
 u_{xxx} \left(-6 a_{12} \left(a_{11}+a_{22}\right)
   \left(a_{12}^2+a_{22}^2\right){}^2 u_{xxy}+6
   \left(-a_{12}^4+a_{22}^2 a_{12}^2+4 a_{11} a_{22}
   a_{12}^2+a_{11}^2 \left(a_{12}^2-a_{22}^2\right)\right)
   \left(a_{12}^2+a_{22}^2\right) u_{xyy}+\right.\\
   &\left.+2 a_{12}
   \left(a_{11}+a_{22}\right) \left(3 a_{12}^4-a_{22}^2 a_{12}^2-8
   a_{11} a_{22} a_{12}^2-a_{11}^2 \left(a_{12}^2-3
   a_{22}^2\right)\right)
   u_{yyy}\right)+\left(a_{12}^2+a_{22}^2\right){}^3
   u_{xxx}^2+ \\
   & +u_{xxy} \left(6
   \left(a_{11}^2+a_{12}^2\right) \left(-a_{12}^4+a_{22}^2
   a_{12}^2+4 a_{11} a_{22} a_{12}^2+a_{11}^2
   \left(a_{12}^2-a_{22}^2\right)\right) u_{yyy}-18 a_{12}
   \left(a_{11}^2+a_{12}^2\right) \left(a_{11}+a_{22}\right)
   \left(a_{12}^2+a_{22}^2\right) u_{xyy}\right)+\\
   &+9
   \left(a_{12}^2+a_{22}^2\right){}^2 \left(a_{11}^2+a_{12}^2\right)
   u_{xxy}^2-6 a_{12} \left(a_{11}+a_{22}\right)
   \left(a_{11}^2+a_{12}^2\right){}^2 u_{xyy}
   u_{yyy}+9 \left(a_{12}^2+a_{22}^2\right)
   \left(a_{11}^2+a_{12}^2\right){}^2
   u_{xyy}^2+\left(a_{11}^2+a_{12}^2\right){}^3
   u_{yyy}^2\, ,
\end{align*}
up to   a nonzero factor $\det(A)^{-6}$. Without loss of generality,  the matrix $A$ has been assumed to be symmetric, meaning that
\begin{equation*}
A\cdot A^t=A^2=\left(
\begin{array}{cc}
 a_{11}^2+a_{12}^2 & a_{11} a_{12}+a_{22} a_{12} \\
 a_{11} a_{12}+a_{22} a_{12} & a_{12}^2+a_{22}^2 \\
\end{array}
\right)=\left(\begin{array}{cc}u_{xx} & u_{xy} \\u_{xy} & u_{yy}\end{array}\right)\, .
\end{equation*}
Taking the last relation into account, one finally obtains  $F=0$, where $F$ is given by \eqref{eqFormulaSospirata}.
%
%
\subsection{The case when the fiducial hypersurface is the hyperboloid $S_-$}\label{subsecLorentz}
If we set now $Q=Q_-$, then
the  $\Aff(3)$--invariant PDE $\E_\Sigma$ constructed by Theorem \ref{corCasoAff} will be a subset of $J^3_-$.
%
In analogy with \eqref{eqTrFreeCubQUO} we identify
\begin{equation*}
S_0^3\R^{2\,\ast}=\frac{\Span{x^3, 3x^2y, 3xy^2, y^3}}{\Span{x^3-xy^2,yx^2-y^3}}
\end{equation*}
 with the subspace $ \Span{x^3, y^3}$,
that we keep denoting by the same symbol $S_0^3\R^{2\,\ast}$.

\subsubsection{The action of $\CO(1,1)$ on $ \Span{x^3, y^3}$}
By employing hyperbolic rotations, the same technique used in Section~\ref{secCasoRiem} allows to show that
%
the action of $\CO(1,1)$ on $S_0^3\R^{2\,\ast}$ can be identified with the standard one on $\R^2$; however, differently from the case treated in Section \ref{subsecRieman}, in addition to $\overline{\Sigma}=\{0\}$, here we have a codimension--one $\OOO(1,1)$--invariant  submanifold, namely the (degenerate) quadric
\begin{equation*}
\overline{\Sigma}:=\{\xi^2-\eta^2=0\}\, ,
\end{equation*}
where $\xi$ and $\eta$ are defined by \eqref{eqxieta}.
As a manifold, $\overline{\Sigma}$ is singular, since the quadric $\xi^2-\eta^2=0$ is the union of the lines $\xi-\eta=0$ and $\xi+\eta=0$. Nevertheless, it will be easier to work with the whole quadric $\xi^2-\eta^2=0$, rather than with each factor separately.

\subsubsection{The  PDE associated with $\overline{\Sigma}:=\{\xi^2-\eta^2=0\}$}
It suffices to observe that, by the very definition of $\xi$ and $\eta$,
\begin{equation*}
\xi=u_{xxx}+3u_{xyy}\, ,\quad \eta=u_{yyy}+3u_{xxy}\, ,
\end{equation*}
whence   the $H^{(2)}$--invariant hypersurface $(\E_\Sigma)_{a_-^{2}}$ of $J^3_{a_-^{2}}$ associated with $\Sigma$ reads
\begin{equation}\label{eqDaRichiamareAllaFine1}
f(u_{xxx},u_{xxy},u_{xyy},u_{yyy})=(u_{xxx}+3u_{xyy})^2-(u_{yyy}+3u_{xxy})^2=0\, .
\end{equation}
As before, we use   $\GL(2)$, that acts naturally on $J^2_{a_-^{1}}=\GL(2)\cdot Q$, to bring   any point $(u_{xx},u_{xy},u_{yy})\in J^2_{a_-^{1}}$   to the form
\begin{equation*}
A\cdot \left(\begin{array}{cc}1 & 0 \\0 & -1\end{array}\right) \cdot A^t\, ,
\end{equation*}
 where    $A\in\GL(2)\subset H^{(2)}$.  The formulas for obtaining $F$ out of $f$ are analogous to the ones employed before and we omit them.  Surprisingly enough, the resulting PDE $F=0$ will be given exactly by the same function $F$ as above, that is, \eqref{eqFormulaSospirata}.


\section{Compatibility conditions and solutions in the case of the $\Aff(3)$--invariant system of PDEs}\label{sec:det.hess.maggiore.zero}

A system of PDEs whose number of equations is strictly greater than the number of unknown functions could be \emph{overdetermined}, i.e., it could possess a certain number of non--trivial compatibility conditions \cite{MR1240056}.

\smallskip
For instance, 
in the case when $\det(\Hess(u))>0$ we have obtained\footnote{In fact,  as we stressed in    Section \ref{subsecLorentz}, the same system  has been obtained  in the case $\det(\Hess(u))<0$ as well.}  the system \eqref{eqn:sys.to.be.studied} of two PDEs in one unknown function $u=u(x,y)$: we will see that it is indeed an overdetermined system.
%
Before computing 
the compatibility conditions of the system \eqref{eqn:sys.to.be.studied}, let us observe that functions $u=u(x,y)$, with either $u_{xx}=0$ or $u_{yy}=0$, are solutions to the system, so that
it makes  sense to assume, from now on, that both
$u_{xx}$ and $u_{yy}$ be nonzero.
Thus, we write system \eqref{eqn:sys.to.be.studied} as follows:
\begin{equation}\label{eqn:system.conv.form}
\E_+:
\left\{
\begin{array}{l}
u_{xxx}=-\frac{u_{xx}(2u_{xy}u_{yyy}-3u_{xyy}u_{yy})}{u_{yy}^2}\,,
\\
u_{xxy}=\frac{u_{xx}u_{yy}u_{yyy}-4u_{xy}^2u_{yyy}+6u_{xy}u_{xyy}u_{yy}}{3u_{yy}^2}\,.
\end{array}
\right.
\end{equation}
\subsection{Compatibility conditions and solutions of the system of PDEs \eqref{eqn:system.conv.form}}
The technique for obtaining the compatibility conditions is standard.
First, one takes the total derivatives of both sides of both the equations  the  system \eqref{eqn:system.conv.form} consists of: this allows   to express $u_{xxxx}$, $u_{xxxy}$ and $u_{xxyy}$ in terms of $u_{xyyy}$, $u_{yyyy}$, $u_{xyy}$, $u_{yyy}$ and second order derivatives; then one observes that the cross total differentiation $D_y(u_{xxx})-D_x(u_{xxy})$, that is
$$
-\frac{8}{3}\frac{(u_{xy}u_{yy}u_{yyyy}-2u_{xy}u_{yyy}^2+2u_{xyy}u_{yy}u_{yyy}-u_{xyyy}u_{yy}^2)(u_{xx}u_{yy}-u_{xy}^2)}{u_{yy}^4}\,,
$$
 has to vanish. Since we are assuming $\det(\Hess(u))\neq 0$ and $u_{yy}\neq 0$, the vanishing of the above expression  yields a fourth--order compatibility condition:
\begin{equation}\label{eqn:cond.4.ord}
    u_{xy}u_{yy}u_{yyyy}-2u_{xy}u_{yyy}^2+2u_{xyy}u_{yy}u_{yyy}-u_{xyyy}u_{yy}^2=0\,.
\end{equation}
Since \eqref{eqn:cond.4.ord} allows to express also $u_{xyyy}$ in terms of $u_{yyyy}$, $u_{xyy}$, $u_{yyy}$ and derivatives of second order, we find out that, in view of what already seen above, all but one fourth--order derivatives, i.e.,   $u_{xxxx}$, $u_{xxxy}$, $u_{xxyy}$ and $u_{xyyy}$,  can be expressed in terms of the remaining fourth--order derivative, i.e.,  $u_{yyyy}$, and of the third--order derivatives $u_{xyy}$, $u_{yyy}$, and of the  derivatives of second order as well.
This scheme repeats verbatim for fifth--order compatibility conditions.
Indeed, taking into account the relations that we have obtained so far, the fifth--order derivatives $u_{xxxxx}$, $u_{xxxxy}$, $u_{xxxyy}$, $u_{xxyyy}$, $u_{xyyyy}$ can be expressed in terms of $u_{yyyyy}$, $u_{yyyy}$, $u_{xyy}$, $u_{yyy}$, and of derivatives of second order: the vanishing of the cross total differentiation $D_y(u_{xxxx})-D_x(u_{xxxy})$,  that is
\begin{equation*}
\frac{1}{9}\frac{u_{xx}}{u_{yy}^5}(u_{xx}u_{yy}-u_{xy}^2)(9u_{yy}^2u_{yyyyy}-45u_{yy}u_{yyy}u_{yyyy}+40u_{yyy}^3)=0\, ,
\end{equation*}
gives then
\begin{equation}\label{eqn:conic.yu}
9u_{yy}^2u_{yyyyy}-45u_{yy}u_{yyy}u_{yyyy}+40u_{yyy}^3=0\, ,
\end{equation}
since both   $u_{xx}$ and  $u_{yy}$ are nonzero,  and we are assuming $\det(\Hess(u))\neq 0$. These results have been obtained also in \cite{Arnal.Valiquette} by using different techniques. Let us stress that \eqref{eqn:conic.yu} is actually a well-known ODE: its solution space is made of the conics in the $(y,u)$--plane, see, e.g., \cite{MR3253544,MR1240056}. In fact the following result holds.
\begin{proposition}\label{propProposizioncina}
The function $u=u(x,y)$ is a solution to \eqref{eqn:conic.yu} if and only if
\begin{equation*}
a(x)y^2+b(x)yu+c(x)u^2+d(x)y+e(x)u+f(x)=0\,.
\end{equation*}
\end{proposition}

Now we wonder about higher--order compatibility conditions: since
 equation~\eqref{eqn:conic.yu} allows to obtain the fifth--order derivative $u_{yyyyy}$ in terms of those of lower order, in view of all we have obtained so far, it turns out that all fifth--order derivatives can be expressed in terms of the lower--order ones, forming a system of six equations. Such system is not overdetermined because   a direct computation shows that it gives no other compatibility conditions.



%

Since the system  \eqref{eqn:sys.to.be.studied} is  symmetric with respect to the interchanging of $x$ and $y$, the consequences \eqref{eqn:cond.4.ord} and \eqref{eqn:conic.yu} have to still hold true after switching $x$ with $y$, i.e., we get the system:
\begin{equation}\label{eqn:4.and.5.cond}
\left\{
\begin{array}{l}
    u_{xy}u_{yy}u_{yyyy}-2u_{xy}u_{yyy}^2+2u_{xyy}u_{yy}u_{yyy}-u_{xyyy}u_{yy}^2=0\,,
    \\
    \\
    u_{xy}u_{xx}u_{xxxx}-2u_{xy}u_{xxx}^2+2u_{xxy}u_{xx}u_{xxx}-u_{xxxy}u_{xx}^2=0\,,
    \\
    \\
    9u_{yy}^2u_{yyyyy}-45u_{yy}u_{yyy}u_{yyyy}+40u_{yyy}^3=0\,,
    \\
    \\
    9u_{xx}^2u_{xxxxx}-45u_{xx}u_{xxx}u_{xxxx}+40u_{xxx}^3=0\,.
\end{array}
\right.
\end{equation}
%
Proposition~\ref{propProposizioncina} says precisely that the solutions to the sub--system of \eqref{eqn:4.and.5.cond} made of the last two equations are those conics in the $(y,u)$--plane (with  coefficients depending upon $x$) that are simultaneously conics  in the $(x,u)$--plane (with  coefficients depending upon $y$): in other words, $u=u(x,y)$ is a solution to the aforementioned sub--system of \eqref{eqn:4.and.5.cond} if and only if
\begin{equation}\label{eqn:u.conic}
au^2 + (h_0+h_1x+h_2y+h_3xy)u + k_0+k_1x+k_2y+k_3x^2+k_4xy+k_5y^2+k_6x^2y+k_7xy^2+k_8x^2y^2=0\,,\,\,a,h_i,k_i\in\mathbb{R}\,.
\end{equation}
%
%
%
%


%


\subsection{A tensorial (re)formulation of the compatibility conditions}
In this section we show that the compatibility conditions can be given a tensorial interpretation, much as we did in Section~\ref{secTensDer} above for the system~\eqref{eqn:system.conv.form} itself: this can be achieved by
considering the prolongation to higher--order jet spaces of a system of PDEs, equipped with its compatibility conditions, and the higher--order Hessian.
To begin with, we interpret the
system \eqref{eqn:system.conv.form} as a 10--dimensional submanifold of $J^3$: local coordinates on $\mathcal{S}$ are
$$
x,y,u,u_x,u_y,u_{xx},u_{xy},u_{yy},u_{xyy},u_{yyy}\,.
$$
The  fourth--order compatibility conditions  \eqref{eqn:cond.4.ord}, together with the original system \eqref{eqn:system.conv.form} and the outcomes of its total differentiation, describes an 11--dimensional submanifold of $J^4$, which we denote by $\E_+^{(I)}$: indeed, similarly as we have  showed above, the eleven functions
\begin{equation}\label{eqn:coord.on.S1}
x,y,u,u_x,u_y,u_{xx},u_{xy},u_{yy},u_{xyy},u_{yyy},u_{yyyy}
\end{equation}
can be taken as  local coordinates on $\E_+^{(I)}$, since $u_{xxx}$, $u_{xxy}$, $u_{xxxx}$, $u_{xxxy}$, $u_{xxyy}$ and $u_{xyyy}$ can be expressed in terms of \eqref{eqn:coord.on.S1}.
This very reasoning can be repeated at the fifth order: we consider the system made of the equation \eqref{eqn:conic.yu}, the system describing the submanifold $\E_+^{(I)}$, together with the outcomes of its total differentiation, and interpret  this new system as a submanifold of $J^5$, which will be accordingly denoted  by $\E_+^{(II)}$. Also the submanifold $\E_+^{(II)}$ turns out to be 11--dimensional since, as we have seen above, all fifth--order derivatives can be expressed in terms of lower--order ones and then  we can take again   \eqref{eqn:coord.on.S1} as  local coordinates on $\E_+^{(II)}$.
The following result, that can be obtained by direct computations, generalizes the fact that system \eqref{eqn:system.conv.form} can be obtained by requiring that
$u_{xxx}dx^3+3u_{xxy}dx^2dy+3u_{xyy}dxdy^2+u_{yyy}dy^3$
be proportional to
$u_{xx}dx^2+2u_{xy}dxdy+u_{yy}dy^2$, see \eqref{eqn:H3.prop.H2.H1.1}.
\begin{proposition}
The following tensorial relations
\begin{eqnarray*}
u_{xxx}dx^3+3u_{xxy}dx^2dy+3u_{xyy}dxdy^2+u_{yyy}dy^3 &\propto& u_{xx}dx^2+2u_{xy}dxdy+u_{yy}dy^2\,,\\
u_{xxxx}dx^4+4u_{xxxy}dx^3dy+6u_{xxyy}dx^2dy^2+u_{xyyy}dxdy^3+u_{yyyy}dy^4 &\propto& u_{xx}dx^2+2u_{xy}dxdy+u_{yy}dy^2\,,\\
u_{xxxxx}dx^5+5u_{xxxxy}dx^4dy+10u_{xxxyy}dx^3dy^2+\dots+u_{yyyyy}dy^5 &\propto& u_{xx}dx^2+2u_{xy}dxdy+u_{yy}dy^2\,,
\end{eqnarray*}
hold true, once restricted, respectively, to $\E_+$, $\E_+^{(I)}$ and $\E_+^{(II)}$.
\end{proposition}

%

\section{The geometry of 
characteristics lines in the case  of $\Aff(3)$--invariant scalar PDEs}\label{sec:det.hess.minore.zero}

A particularly simple and geometrically well--behaved class of scalar PDEs in two independent variables is made of those whose characteristic conic  distribution degenerates to a vectorial one, see~\cite{MR2985508,MR2503974,MR1194520,MR1504329,MannoMoreno2016}: such PDEs are called also of \emph{Goursat type}. 
%
%
%
%
%
In this section we will show  that the $3\Rd$--order scalar PDEs given by  \eqref{eqn:union.PDEs} are precisely of this kind: in particular, we shall see that their  \emph{symbol} has rank one and that they  are completely determined by a $3$--dimensional vector sub--distribution of the $2\Nd$--order contact distribution $\CC^2$ on $J^2$ (cf.~\eqref{eqn:higher.contact}), called the \emph{characteristic distribution}.

For the sake of simplicity, here we will deal only with one of the two factors of~\eqref{eqn:union.PDEs}, for instance, with
\begin{multline}\label{eqn:union.PDEs.1}
\E^1_-:\,\,\sqrt{-\det(\Hess(u))} \big(u_{xx}u_{yy}u_{yyy}-3u_{xxy}u_{yy}^2-4u_{xy}^2u_{yyy}+6u_{xy}u_{xyy}u_{yy} \big)
\\
+ \big(-3u_{xx}u_{xy}u_{yy}u_{yyy}+3u_{xx}u_{xyy}u_{yy}^2-u_{xxx}u_{yy}^3
+3u_{xxy}u_{xy}u_{yy}^2+4u_{xy}^3u_{yyy}-6u_{xy}^2u_{xyy}u_{yy}\big)=0\, ,
\end{multline}
since  the whole machinery applies  as well to the other factor; we stress also that we work all the time over $J^2_-$. 
The ($3$--dimensional) characteristic distribution $\mathcal{V}$ of \eqref{eqn:union.PDEs.1} will be studied below.

\subsection{Characteristic lines  of a $3\Nd$--order scalar PDE}
The construction of characteristic lines of 
a $3\Rd$ order Monge--Amp\`ere equation in two independent variables of Goursat type is  explained in~\cite[Section~5]{MannoMoreno2016}, and it goes as follows.\par
A scalar
$3\Rd$ order PDE in one unknown function $u=u(x^1,x^2)=u(x,y)$ and $2$ independent variables $(x^1,x^2)=(x,y)$ is locally described by 
\begin{equation}\label{eqn:PDE.3.order.scalar}
\E=\{F(x^i,u,u_i, u_{ij}, u_{ijk})=0\}\,.
\end{equation}
We can assume that \eqref{eqn:PDE.3.order.scalar} be a hypersurface of $J^3$. 
%
The departing point of our construction is the symmetric $3$--form
\begin{equation}\label{eq:symbol.3}
\sum_{i\leq j\leq k}\frac{\partial F}{\partial u_{ijk}}\eta_i\eta_j\eta_k\, ,
\end{equation}
that can be associated with the function $F$,
called  the \emph{symbol} of $F$ (see, e.g.,~\cite{GMMS_ADE,FG,Petrovsky1992}), and 
the symmetric $3$--tensor
\begin{equation}\label{eq:symbol.3.as.tensor}
\sum_{i\leq j\leq k}\frac{\partial F}{\partial u_{ijk}}D^{(3)}_{x^i}D^{(3)}_{x^j}D^{(3)}_{x^k}
\end{equation}
belonging to $\TT^{3}\odot \TT^{3}\odot \TT^{3}$, where $\TT^3$ is the tautological bundle over $J^3$ (see Definition \ref{defTautBundle}) and $D^{(3)}_{x^i}$ denotes the total derivative operator with respect to $x^i$, truncated to the third order.


\par
The symmetric $3$--form \eqref{eq:symbol.3} (as well as \eqref{eq:symbol.3.as.tensor}), once restricted to the PDE $\E$ given by \eqref{eqn:PDE.3.order.scalar}, gives the so--called \emph{symbol} of $\E$.
Even if the symbol of $\E$ is defined up to a non--vanishing factor, we will be using only its properties that  do not depend on this factor.

\par
Let us assume that
the symbol of $\E$  be a perfect cube, i.e.,  the cubic power of a linear factor: then there must exist functions $f^i$   on $\E$, such that
\begin{equation}\label{eq:symbol.3bis}
\left.\left(\sum_{i\leq j\leq k}\frac{\partial F}{\partial u_{ijk}}\eta_i\eta_j\eta_k\right)\right|_{\E} = (f^i\eta_i)^3\, .
\end{equation}
%
Since the symbol  of $\E$ satisfies  \eqref{eq:symbol.3bis}, it  is called a \emph{rank--one} symbol: from a tensorial viewpoint, the restriction of \eqref{eq:symbol.3.as.tensor} to $\E$ breaks into the $3{\Rd}$ power of a single $1$--tensor:
\begin{equation*}
\left.\left(\sum_{i\leq j\leq k}\frac{\partial F}{\partial u_{ijk}}D^{(3)}_{x^i}D^{(3)}_{x^j}D^{(3)}_{x^k}
\right)\right|_{\E}=\left(\left.f^iD^{(3)}_{x^i}\right|_{\E}\right)^3\,.
\end{equation*}
This allows   to associate, with each point $a^3\in\E_{a^2}=\E\cap J^3_{a^2}$ (cf. \eqref{eqn:fibre.of.E}),  the tangent line at $a^2$
\begin{equation}\label{eq:char.line}
    l_{a^3}=\Span{f^i(a^3)D^{(3)}_{x^i}\big|_{a^3}}\,,
\end{equation}
called a \emph{characteristic line of $\E$ at $a^3$}.

%
%
%

%
\subsection{The $3$--dimensional distribution associated with the $\Aff(3)$--invariant PDE with $\det(\Hess(u))<0$}

In this section, to keep the notation light, we will denote a point of $J^k_-$ simply by $a^k$.\par

Since the symbol of the PDE \eqref{eqn:union.PDEs.1} turns out to be of rank one, being  proportional to
\begin{equation*}
    \left(\sqrt{-\det(\Hess(u))}\,u_{yy}\,\,\eta_1-\left(\sqrt{-\det(\Hess(u))}\,u_{xy}+u_{xx}u_{yy}-u_{xy}^2\right)\,\eta_2\right)^3\, ,
\end{equation*}
it makes sense to consider the line~\eqref{eq:char.line}   for each point $a^3=(x,y,u,u_i,u_{ij},u_{ijk})$ that satisfies  \eqref{eqn:union.PDEs.1}:
\begin{equation}\label{eqn:char.line}
    l_{a^3}=\Span{\sqrt{-\det(\Hess(u))}\,u_{yy}\,\,D^{(3)}_x-\left(\sqrt{-\det(\Hess(u))}\,u_{xy}+u_{xx}u_{yy}-u_{xy}^2\right)\,D^{(3)}_y }\, ,
\end{equation}
where
\begin{eqnarray*}
&D^{(3)}_x=\partial_x+u_x\partial_u+u_{xx}\partial_{u_x}+u_{xy}\partial_{u_y}+u_{xxx}\partial_{u_{xx}}
+u_{xxy}\partial_{u_{xy}} +u_{xyy}\partial_{u_{yy}}\, , \\
&D^{(3)}_y=\partial_y+u_y\partial_u+u_{xy}\partial_{u_x}+u_{yy}\partial_{u_y}+u_{xxy}\partial_{u_{xx}}
+u_{xyy}\partial_{u_{xy}} +u_{yyy}\partial_{u_{yy}} \, ,
\end{eqnarray*}
and  $u_{xxx}$ is obtained from \eqref{eqn:union.PDEs.1}. \par
Let us recall that $\CC^2$ denotes the $2\Nd$ order contact distribution on $J^2$, cf. \eqref{eqn:higher.contact}. 
We are then  in position of defining a conic
subset $\mathcal{V}_{a^2}$ of $\CC^2_{a^2}$ at $a^2=(x,y,u,u_i,u_{ij})\in J^2_-$:   for each point $a^3=(x,y,u,u_i,u_{ij},u_{ijk})$ of the fiber $(\E^1_-)_{a^2}$ of the PDE \eqref{eqn:union.PDEs.1} over  $a^2=(x,y,u,u_i,u_{ij})$, we consider the line $l_{a^3}$ given by \eqref{eqn:char.line} and we let
\begin{equation}\label{eq:3d.dist}
    \mathcal{V}_{a^2}:=\bigcup_{a^3\in (\E_-^1)_{a^2}}  l_{a^3}\,,
\end{equation}
i.e., as the point $a^3$ describes $(\E_-^1)_{a^2}$, the corresponding line $l_{a^3}$ sweeps the conic subset
$\mathcal{V}_{a^2}$.
\begin{proposition}\label{prop:distr.3ordMA}
The distribution of conic subsets of $\CC^{2}$
\begin{equation}\label{eqn.conic.vector.distr}
J^2_-\ni a^2\longmapsto \mathcal{V}_{a^2}\subset\CC^2_{a^2}
\end{equation}
defined by ~\eqref{eq:3d.dist} 
turns out to be a vector distribution locally given by
\begin{equation}\label{eq:forma.esplicita.dist}
    \mathcal{V}=\langle -u_{yy}D^{(2)}_x+bD^{(2)}_y\,,\,2b\partial_{u_{xx}}+u_{yy}\partial_{u_{xy}} \,,\, -b^2\partial_{u_{xx}} +u_{yy}^2\partial_{u_{yy}}\rangle\,,
\end{equation}
where
$$
b=u_{xy}-\sqrt{-\det(\Hess(u))}\,.
$$
\end{proposition}
\begin{proof}
A direct (and tedious) computation shows that $\mathcal{V}_{a^2}$ is locally described by
\begin{equation}\label{eqn:char.var}
\mathcal{V}_{a^2}:\left\{
\begin{array}{l}
y^1u_{xy}\sqrt{-\det(\Hess(u))}+y^2u_{yy}\sqrt{-\det(\Hess(u))}+y^1u_{xx}u_{yy}-y^1u_{xy}^2=0\,,
\\
\\
2p_{12}u_{yy}\sqrt{-\det(\Hess(u))}-2\sqrt{-\det(\Hess(u))}p_{22}u_{xy}+p_{11}u_{yy}^2-2p_{12}u_{xy}u_{yy}-p_{22}u_{xx}u_{yy}+2p_{22}u_{xy}^2=0\,,
\end{array}
\right.
\end{equation}
where $y^1,y^2,p_{11},p_{12},p_{22}$ are local coordinates on $\CC^2_{a^2}$ with respect to the basis $D^{(2)}_x|_{a^2},D^{(2)}_y|_{a^2},\partial_{u_{11}}|_{a^2},\partial_{u_{12}}|_{a^2},\partial_{u_{22}}|_{a^2}$ of $\CC^2_{a^2}$. 
In other words, for each point $a^2=(x,y,u,u_i,u_{ij})\in J^2_-$, the subset $\mathcal{V}_{a^2}$ of $\CC^2_{a^2}$ is described by the  system~\eqref{eqn:char.var} of two (independent) linear equations in $y^1,y^2,p_{11},p_{12},p_{22}$: this means that the correspondence \eqref{eqn.conic.vector.distr}
defines a $3$--dimensional linear distribution on $J^2_-$ inscribed in the $2\Nd$ order contact distribution $\CC^2$. The space of solutions of the system~\eqref{eqn:char.var}  is precisely ~\eqref{eq:forma.esplicita.dist}.
\end{proof}
%
%
%
%
%
For the sake of paper's self--consistency, we explain how PDE~\eqref{eqn:union.PDEs.1} can be actually recovered out of $\mathcal{V}$.
%
%
Recalling that the PDE \eqref{eqn:union.PDEs.1} can be interpreted as a hypersurface of $J^3$, 
one can use $\mathcal{V}$  to define   the subset
\begin{equation}\label{eqn:Goursat.considerata}
\E_{\mathcal{V}} := \left\{a^3\in J^3_-\,\,|\,\, \TT^3_{a^3}\cap \mathcal{V}_{a^2}\neq 0\,,\,\,a^2\in J^2_-\,,\,\,\pi_{3,2}(a^3)=a^2\right\}
\end{equation}
of $J^3_-$, where $\TT^3$ is the tautological bundle on $J^3$, see Definition \ref{defTautBundle}. Since $\TT^3_{a^3}$ and $\mathcal{V}_{a^2}$ are subspaces of dimension $2$ and $3$, respectively, of the $5$--dimensional vector space $\CC^2_{a^2}$,  their intersection is generically $0$--dimensional: the condition that such intersection be non--trivial reads
\begin{equation*}
\det
\begin{pmatrix}
1 & 0 & u_{xxx} & u_{xxy} & u_{xyy}
\\
0 & 1 & u_{xxy} & u_{xyy} & u_{yyy}
\\
-u_{yy} & b & 0 & 0 & 0
\\
0 & 0 & 2b & u_{yy} & 0
\\
0 & 0 & -b^2 & 0 & u_{yy}^2
\end{pmatrix}
=0\,,
\end{equation*}
which gives exactly the equation \eqref{eqn:union.PDEs.1}. In other words, we have proved that $\E^1_-=\E_{\mathcal{V}}$, i.e., $\E^1_-$ can be recovered out of $\mathcal{V}$.
\par
Another approach is based on the $\Aff(3)$--invariance of the equation~\eqref{eqn:union.PDEs.1}. If we fix the point
\begin{equation}\label{eqn:a2.000}
a^2=(0,0,0,0,0,-1,0,1)\in J^2_-\,,
\end{equation}
then the equation~\eqref{eqn:union.PDEs.1} becomes (cf. \eqref{eqEffeUnoEffeDue} and \eqref{eqDaRichiamareAllaFine1})
\begin{equation}\label{appoggio1}
u_{xxx}+3u_{xxy}+3u_{xyy}+u_{yyy}=0\, ,
\end{equation}
whose symbol is simply
$
(\eta_1+\eta_2)^3
$, whence the
characteristic line is
$
D^{(3)}_x+D^{(3)}_y
$,
with $u_{xxx}=-(3u_{xxy}+3u_{xyy}+u_{yyy})$ in the expression of  $D^{(3)}_x$, cf. \eqref{eqEffeUnoEffeDue} and \eqref{eqDaRichiamareAllaFine1}. This immediately leads to the $3$--dimensional linear subspace
\begin{equation}\label{eq:dist.in.un.punto}
   V:=\Span{D^{(2)}_x\big|_{a^2}+D^{(2)}_y\big|_{a^2}\,,\,-2\partial_{u_{xx}}\big|_{a^2}+\partial_{u_{xy}}\big|_{a^2}\,,\,-\partial_{u_{xx}}\big|_{a^2}+\partial_{u_{yy}}\big|_{a^2}  }
\end{equation}
of $\CC^2_{a^2}$, and  it is easy to see  that~\eqref{eq:dist.in.un.punto} gives precisely  the evaluation at the point $a^2$ of
\eqref{eq:forma.esplicita.dist}.\par

\smallskip
If we now compute the fiber
\begin{equation*}
    (\E_{\mathcal{V}})_{a^2}=\{ a^3\in J^3_-\mid \TT^3_{a^3}\cap V\neq 0\}
\end{equation*}
of the equation  \eqref{eqn:Goursat.considerata} at the point \eqref{eqn:a2.000},
where $V$ is given by~\eqref{eq:dist.in.un.punto},
we find out that such a fiber is given by~\eqref{appoggio1}. In other words, the equation~\eqref{eqn:union.PDEs.1} and the
Goursat--type Monge--Amp\`ere equation~\eqref{eqn:Goursat.considerata}, that are both $\Aff(3)$--invariant, have the same fiber~\eqref{appoggio1}  over the point $a^2\in J^2_-$: this means that they must coincide.


\subsection{The symbol of~\eqref{eqFormulaSospirata}  as a feature of  the Blaschke metric}

Now we go back to the whole equation $\E$, that is the zero locus in $J^3$ of the function $F$ given by \eqref{eqFormulaSospirata}: by direct computations, that we will omit, it is possible to establish  equation~\eqref{eqn:Pick.and.Sym} below, that  allows to obtain the symbol of $F$ directly from the Blaschke metric and the Fubini--Pick cubic form (see Section \ref{subAffCaseR3}).

\begin{proposition}
Let $h=h_{ij}$ be the Blaschke metric and $C=C_{abc}$ the Fubini--Pick cubic form: then\footnote{The summation at the left--hand side of \eqref{eqn:Pick.and.Sym} is over all indices.}
\begin{equation}\label{eqn:Pick.and.Sym}
8(|\det(\Hess(u))|)^{\frac52}h^{ai}h^{bj}h^{ck}C_{abc}\eta_i\eta_j\eta_k =\mathrm{sgn}(\det(\Hess(u)))\sum_{i\leq j\leq k}\frac{\partial F}{\partial u_{ijk}}\eta_i\eta_j\eta_k\,,
\end{equation}
where $F$ is given by \eqref{eqFormulaSospirata}.
\end{proposition}
%
It should be stressed that the cubic form at the right--hand side of~\eqref{eqn:Pick.and.Sym} is the symbol of the \emph{function} $F$ and not of the equation $\E=\{F=0\}$ since, as we have seen above, for $\det(\Hess(u))>0$, the equation is actually a system of two equations. Nevertheless, for $\det(\Hess(u))<0$, we have that $F=f^1f^2$, where $f^i$ are the two functions given by the left-hand side of \eqref{eqn:union.PDEs} and it turns out that
$$
\sum_{i\leq j\leq k}\frac{\partial F}{\partial u_{ijk}}\eta_i\eta_j\eta_k =
\sum_{i\leq j\leq k}\frac{\partial (f^1f^2)}{\partial u_{ijk}}\eta_i\eta_j\eta_k
=\sum_{i\leq j\leq k}f^2\frac{\partial f^1}{\partial u_{ijk}}\eta_i\eta_j\eta_k + \sum_{i\leq j\leq k}f^1\frac{\partial f^2}{\partial u_{ijk}}\eta_i\eta_j\eta_k\,.
$$





\bibliographystyle{abbrvnat}
\bibliography{BibUniver}

\def\cprime{$'$} \def\scr{\mathcal}
  \def\polhk#1{\setbox0=\hbox{#1}{\ooalign{\hidewidth
  \lower1.5ex\hbox{`}\hidewidth\crcr\unhbox0}}}
\begin{thebibliography}{26}
\providecommand{\natexlab}[1]{#1}
\providecommand{\url}[1]{\texttt{#1}}
\expandafter\ifx\csname urlstyle\endcsname\relax
  \providecommand{\doi}[1]{doi: #1}\else
  \providecommand{\doi}{doi: \begingroup \urlstyle{rm}\Url}\fi

\bibitem[Alekseevsky et~al.((to appear))Alekseevsky, Manno, and
  Moreno]{AMMv2_2022}
D.~V. Alekseevsky, G.~Manno, and G.~Moreno.
\newblock Projectively and affinely invariant pdes on hypersurfaces.
\newblock \emph{Proceedings of the Edinburgh Mathematical Society}, (to
  appear).

\bibitem[Alekseevsky et~al.(2012)Alekseevsky, Alonso-Blanco, Manno, and
  Pugliese]{MR2985508}
D.~V. Alekseevsky, R.~Alonso-Blanco, G.~Manno, and F.~Pugliese.
\newblock Contact geometry of multidimensional {M}onge-{A}mp\`ere equations:
  characteristics, intermediate integrals and solutions.
\newblock \emph{Ann. Inst. Fourier (Grenoble)}, 62\penalty0 (2):\penalty0
  497--524, 2012.
\newblock ISSN 0373-0956.
\newblock \doi{10.5802/aif.2686}.
\newblock URL \url{http://dx.doi.org/10.5802/aif.2686}.

\bibitem[Alekseevsky et~al.(2021)Alekseevsky, Gutt, Manno, and
  Moreno]{alekseevsky2020general}
D.~V. Alekseevsky, J.~Gutt, G.~Manno, and G.~Moreno.
\newblock A general method to construct invariant {PDEs} on homogeneous
  manifolds.
\newblock \emph{Communications in Contemporary Mathematics}, page 2050089, Jan.
  2021.
\newblock \doi{10.1142/s0219199720500893}.
\newblock URL \url{https://doi.org/10.1142/s0219199720500893}.

\bibitem[Alonso~Blanco et~al.(2008)Alonso~Blanco, Manno, and
  Pugliese]{MR2383541}
R.~Alonso~Blanco, G.~Manno, and F.~Pugliese.
\newblock Contact relative differential invariants for non generic parabolic
  {M}onge-{A}mp\`ere equations.
\newblock \emph{Acta Appl. Math.}, 101\penalty0 (1-3):\penalty0 5--19, 2008.
\newblock ISSN 0167-8019.
\newblock \doi{10.1007/s10440-008-9204-8}.
\newblock URL \url{http://dx.doi.org/10.1007/s10440-008-9204-8}.

\bibitem[Alonso-Blanco et~al.(2009)Alonso-Blanco, Manno, and
  Pugliese]{MR2503974}
R.~Alonso-Blanco, G.~Manno, and F.~Pugliese.
\newblock Normal forms for {L}agrangian distributions on 5-dimensional contact
  manifolds.
\newblock \emph{Differential Geom. Appl.}, 27\penalty0 (2):\penalty0 212--229,
  2009.
\newblock ISSN 0926-2245.
\newblock \doi{10.1016/j.difgeo.2008.06.019}.
\newblock URL \url{https://doi.org/10.1016/j.difgeo.2008.06.019}.

\bibitem[An-Min et~al.(2015)An-Min, Udo, Guosong, and Zejun]{An-Min:2015aa}
L.~An-Min, S.~Udo, Z.~Guosong, and H.~Zejun.
\newblock \emph{Global Affine Differential Geometry of Hypersurfaces}.
\newblock De Gruyter, 2019-07-14T21:27:59.826+02:00 2015.
\newblock ISBN 978-3-11-039090-2.
\newblock \doi{10.1515/9783110268898}.
\newblock URL \url{https://www.degruyter.com/view/product/179593}.

\bibitem[Arnaldsson and Valiquette(2021)]{Arnal.Valiquette}
O.~Arnaldsson and F.~Valiquette.
\newblock Invariants of surfaces in three-dimensional affine geometry.
\newblock \emph{Symmetry, Integrability and Geometry: Methods and
  Applications}, Mar 2021.
\newblock ISSN 1815-0659.
\newblock \doi{10.3842/sigma.2021.033}.
\newblock URL \url{http://dx.doi.org/10.3842/SIGMA.2021.033}.

\bibitem[Boillat(1992)]{MR1194520}
G.~Boillat.
\newblock Sur l'\'equation g\'en\'erale de {M}onge-{A}mp\`ere d'ordre
  sup\'erieur.
\newblock \emph{C. R. Acad. Sci. Paris S\'er. I Math.}, 315\penalty0
  (11):\penalty0 1211--1214, 1992.
\newblock ISSN 0764-4442.

\bibitem[Goursat(1899)]{MR1504329}
E.~Goursat.
\newblock Sur les \'equations du second ordre \`a {$n$} variables analogues \`a
  l'\'equation de {M}onge-{A}mp\`ere.
\newblock \emph{Bull. Soc. Math. France}, 27:\penalty0 1--34, 1899.
\newblock ISSN 0037-9484.
\newblock URL \url{http://www.numdam.org/item?id=BSMF\_1899\_\_27\_\_1\_0}.

\bibitem[Gutt et~al.(2024)Gutt, Manno, Moreno, and Śmiech]{GMMS_ADE}
J.~Gutt, G.~Manno, G.~Moreno, and R.~Śmiech.
\newblock {The moment map on the space of symplectic 3d Monge-Ampère
  equations}.
\newblock \emph{Advances in Differential Equations}, 29\penalty0
  (7/8):\penalty0 575 -- 654, 2024.
\newblock \doi{10.57262/ade029-0708-575}.
\newblock URL \url{https://doi.org/10.57262/ade029-0708-575}.

\bibitem[John(1991)]{FG}
F.~John.
\newblock \emph{Partial Differential Equations}.
\newblock Applied Mathematical Sciences (Book 1). Springer, 1991.

\bibitem[Krasil'shchik et~al.(1986{\natexlab{a}})Krasil'shchik, Lychagin, and
  Vinogradov]{KrasilshchikLychaginVinogradov1986}
I.~S. Krasil'shchik, V.~V. Lychagin, and A.~M. Vinogradov.
\newblock \emph{Geometry of jet spaces and nonlinear partial differential
  equations}, volume~1 of \emph{Advanced Studies in Contemporary Mathematics}.
\newblock Gordon and Breach Science Publishers, New York, 1986{\natexlab{a}}.
\newblock ISBN 2-88124-051-8.

\bibitem[Krasil'shchik et~al.(1986{\natexlab{b}})Krasil'shchik, Lychagin, and
  Vinogradov]{MR861121}
I.~S. Krasil'shchik, V.~V. Lychagin, and A.~M. Vinogradov.
\newblock \emph{Geometry of jet spaces and nonlinear partial differential
  equations}, volume~1 of \emph{Advanced Studies in Contemporary Mathematics}.
\newblock Gordon and Breach Science Publishers, New York, 1986{\natexlab{b}}.
\newblock ISBN 2-88124-051-8.

\bibitem[Kurose(1994)]{Kurose1994ONTD}
T.~Kurose.
\newblock On the divergences of 1-conformally flat statistical manifolds.
\newblock \emph{Tohoku Mathematical Journal}, 46:\penalty0 427--433, 1994.
\newblock URL \url{https://api.semanticscholar.org/CorpusID:122399115}.

\bibitem[Kushner et~al.(2007)Kushner, Lychagin, and Rubtsov]{MR2352610}
A.~Kushner, V.~Lychagin, and V.~Rubtsov.
\newblock \emph{Contact geometry and non-linear differential equations}, volume
  101 of \emph{Encyclopedia of Mathematics and its Applications}.
\newblock Cambridge University Press, Cambridge, 2007.
\newblock ISBN 978-0-521-82476-7; 0-521-82476-1.

\bibitem[Manno et~al.(2007{\natexlab{a}})Manno, Oliveri, and Vitolo]{MR2406036}
D.~Manno, F.~Oliveri, and R.~Vitolo.
\newblock Differential equations uniquely determined by algebras of point
  symmetries.
\newblock \emph{Teoret. Mat. Fiz.}, 151\penalty0 (3):\penalty0 486--494,
  2007{\natexlab{a}}.
\newblock ISSN 0564-6162.
\newblock \doi{10.1007/s11232-007-0069-1}.
\newblock URL \url{http://dx.doi.org/10.1007/s11232-007-0069-1}.

\bibitem[Manno and Moreno(2016)]{MannoMoreno2016}
G.~Manno and G.~Moreno.
\newblock Meta-symplectic geometry of {$3^{\rm rd}$} order {M}onge-{A}mp\`ere
  equations and their characteristics.
\newblock \emph{SIGMA Symmetry Integrability Geom. Methods Appl.}, 12:\penalty0
  032, 35 pages, 2016.
\newblock ISSN 1815-0659.
\newblock \doi{10.3842/SIGMA.2016.032}.
\newblock URL \url{http://dx.doi.org/10.3842/SIGMA.2016.032}.

\bibitem[Manno et~al.(2007{\natexlab{b}})Manno, Oliveri, and Vitolo]{MR2324300}
G.~Manno, F.~Oliveri, and R.~Vitolo.
\newblock On differential equations characterized by their {L}ie point
  symmetries.
\newblock \emph{J. Math. Anal. Appl.}, 332\penalty0 (2):\penalty0 767--786,
  2007{\natexlab{b}}.
\newblock ISSN 0022-247X.

\bibitem[Manno et~al.(2014)Manno, Oliveri, Saccomandi, and Vitolo]{MR3253544}
G.~Manno, F.~Oliveri, G.~Saccomandi, and R.~Vitolo.
\newblock Ordinary differential equations described by their {L}ie symmetry
  algebra.
\newblock \emph{J. Geom. Phys.}, 85:\penalty0 2--15, 2014.
\newblock ISSN 0393-0440.
\newblock \doi{10.1016/j.geomphys.2014.05.028}.
\newblock URL \url{http://dx.doi.org/10.1016/j.geomphys.2014.05.028}.

\bibitem[Nomizu and Sasaki(1994)]{nomizu1994affine}
K.~Nomizu and T.~Sasaki.
\newblock \emph{Affine Differential Geometry: Geometry of Affine Immersions}.
\newblock Cambridge Tracts in Mathematics. Cambridge University Press, 1994.
\newblock ISBN 9780521441773.
\newblock URL \url{https://books.google.it/books?id=lEUVHyjQANcC}.

\bibitem[Olver(1993)]{MR1240056}
P.~J. Olver.
\newblock \emph{Applications of {L}ie groups to differential equations}, volume
  107 of \emph{Graduate Texts in Mathematics}.
\newblock Springer-Verlag, New York, second edition, 1993.
\newblock ISBN 0-387-94007-3; 0-387-95000-1.
\newblock \doi{10.1007/978-1-4612-4350-2}.
\newblock URL \url{http://dx.doi.org/10.1007/978-1-4612-4350-2}.

\bibitem[Petrovsky(1992)]{Petrovsky1992}
I.~G. Petrovsky.
\newblock \emph{Lectures on Partial Differential Equations}.
\newblock Dover Books on Mathematics. Dover Publications, 1992.
\newblock ISBN 978-0486669021.

\bibitem[Saunders(1989)]{MR989588}
D.~J. Saunders.
\newblock \emph{The geometry of jet bundles}, volume 142 of \emph{London
  Mathematical Society Lecture Note Series}.
\newblock Cambridge University Press, Cambridge, 1989.
\newblock ISBN 0-521-36948-7.
\newblock \doi{10.1017/CBO9780511526411}.
\newblock URL \url{http://dx.doi.org/10.1017/CBO9780511526411}.

\bibitem[Ushakov(2000)]{Ushakov2000TheEG}
V.~Ushakov.
\newblock The explicit general solution of trivial monge-amp{\`e}re equation.
\newblock \emph{Commentarii Mathematici Helvetici}, 75:\penalty0 125--133,
  2000.
\newblock URL \url{https://api.semanticscholar.org/CorpusID:121472372}.

\bibitem[Vinogradov(1988)]{Vinogradov1988a}
A.~M. Vinogradov.
\newblock An informal introduction to the geometry of jet spaces.
\newblock \emph{Rend. Sem. Fac. Sci. Univ. Cagliari}, 58:Suppl.:\penalty0
  301--333, 1988.

\bibitem[Yamaguchi(1982)]{MR722524}
K.~Yamaguchi.
\newblock Contact geometry of higher order.
\newblock \emph{Japan. J. Math. (N.S.)}, 8\penalty0 (1):\penalty0 109--176,
  1982.
\newblock ISSN 0289-2316.

\end{thebibliography}

\end{document}